\documentclass[12pt, twoside, leqno]{article} 
\usepackage[latin1]{inputenc} 
\usepackage[english]{babel} 
\usepackage[T1]{fontenc}
\usepackage{bm}       
\usepackage{amsmath,amsthm,amssymb} 
\usepackage{enumerate}    
\usepackage{amsfonts}   
\usepackage{amsthm,booktabs}  
\newcommand{\overbar}[1]{\mkern 1.5mu\overline{\mkern-1.5mu#1\mkern-1.5mu}\mkern 1.5mu}

\newtheorem{thm}{Theorem}[section] 
\newtheorem{prop}[thm]{Proposition}
\newtheorem{mainthm}[thm]{Main Theorem}
 \theoremstyle{definition} 
 \newtheorem{tab}[thm]{Table} 
 \theoremstyle{definition}
\newtheorem{rem}[thm]{Remark}

\begin{document}
\baselineskip=17pt

\title{Cuspidal divisor class groups of non-split Cartan modular curves}
 
\author{Pierfrancesco Carlucci\\
Dipartimento di Matematica\\
 Universit\'a degli Studi di Roma Tor Vergata\\
  Via della Ricerca Scientifica 1, 00133, Rome, Italy\\ 
E-mail: pieffecar@libero.it}
\date{30 April, 2016}

\maketitle


\renewcommand{\thefootnote}{}



\renewcommand{\thefootnote}{\arabic{footnote}}
\setcounter{footnote}{0}
\begin{abstract}
I find an explicit description of modular units in terms of Siegel functions for the modular curves $X^+_{ns}(p^k) $ associated to the normalizer of a non-split Cartan subgroup of level $ p^k $ where $ p\not=2,3 $ is a prime. The Cuspidal Divisor Class Group $  \mathfrak{C}^+_{ns}(p^k) $  on $X^+_{ns}(p^k)$ is explicitly described as a module over the group ring $R = \mathbb{Z}[(\mathbb{Z}/p^k\mathbb{Z})^*/\{\pm 1\}] $.
In this paper I give a formula involving generalized Bernoulli numbers $ B_{2,\chi} $ for $  |\mathfrak{C}^+_{ns}(p^k)| $. 
\end{abstract}

\section{Motivation and overview} 
Let $ X^+_{ns}(n) $ be the modular curve associated to the normalizer of a non-split Cartan subgroup of level $ n$. One noteworthy reason for studying these curves is the Serre's uniformity problem over $ \mathbb{Q} $ stating that there exists a constant $ C >0 $ so that, if $ E $ is an elliptic curve over $ \mathbb{Q} $ without complex multiplication, then the Galois representation:  
$$ \rho_{E,p}: \mbox{Gal}(\overbar{\mathbb{Q}},\mathbb{Q}) \rightarrow \mbox{GL}_2(\mathbb{F}_p) $$ 
attached to the elliptic curve $ E $ is onto for all primes $ p>C$ (see \cite{Serre72} and \cite[pag. 198]{KL}).   
If the Galois representation were not surjective, its image would be contained in one of the maximal proper subgroups
of $ \mbox{GL}_2(\mathbb{F}_p) $. These subgroups are:\\
 
\noindent 1. A Borel subgroup;\\
2. The normalizer of a split Cartan subgroup;\\
3. The normalizer of a non-split Cartan subgroup; \\
4. A finite list of exceptional subgroups. \\

Serre himself showed that if  $ p > 13 $ the image of $ \rho_{E,p} $ is not contained in an exceptional subgroup. Mazur in \cite{Mazur} and Bilu-Parent-Rebolledo in \cite{BiPaRe} presented analogous results for Borel subgroups and split Cartan subgroups respectively. The elliptic curves over $ \mathbb{Q} $ for which the image of the Galois representation is contained in the normalizer of a non-split Cartan subgroup are parametrized by the non-cuspidal rational points of $X^+_{ns}(p) $. Thus the open case of Serre's uniformity problem can be reworded in terms of determining whether there exist $ \mathbb{Q}$-rational points on $X^+_{ns}(p) $, that do not arise from elliptic curves with complex multiplication.
 
This paper focuses on an aspect of the curves $X^+_{ns}(p^k) $ that has never been treated before: their Cuspidal Divisor Class Group $  \mathfrak{C}^+_{ns}(p^k)$, a finite subgroup of the Jacobian $  J^+_{ns}(p^k)$ whose support is contained in the set of cusps of $X^+_{ns}(p^k) $. Let $ \mathfrak{D}_{ns}^+(p^k) $ be the free abelian group generated by the cusps of $X^+_{ns}(p^k) $, let  $ \mathfrak{D}_{ns}^+(p^k)_0 $ be its subgroup consisting of elements of degree 0 and let  $ \mathfrak{F}_{ns}^+(p^k) $ be the group of divisors of modular units of $X^+_{ns}(p^k)$, i.e. those modular functions on $X^+_{ns}(p^k)$ in the modular function field $ F_{p^k} $, which have no zeros and poles in the upper-half plane. We define:
 $$  \mathfrak{C}^+_{ns}(p^k):= \mathfrak{D}_{ns}^+(p^k)_0 / \mathfrak{F}_{ns}^+(p^k).$$
In \cite{Siegelgenerator} Kubert and Lang gave an explicit and complete description of the group of modular units of $ X(p^k) $ in terms of Siegel functions $ g_a(\tau) $ (see \cite{Lang:ef} or \cite{Siegel}) with $a \in \frac{1}{p^k}\mathbb{Z}^2 \setminus \mathbb{Z}^2 $. We will define the set of functions $$ \{ G^+_h(\tau)\}_{h \in  ((\mathbb{Z} / p^k \mathbb{Z})^*/\{\pm 1\})} $$ in terms of classical Siegel functions and we will prove the following result:  \\

\noindent \textbf{Theorem \ref{powprod}}
\textit{If $ p \not= 2,3 $, the group of modular units of the modular curve $ X^+_{ns}(p^k) $ consists (modulo constants) of power products:}
$$  g(\tau)= \prod_{h \in  ((\mathbb{Z} / p^k \mathbb{Z})^*/\{\pm 1\}) }{{G^+_h}^{n^+_h}(\tau)} $$
\textit{where
$$  G^+_h(\tau) = \prod_{t \in ((\mathbf{O}_K / p^k \mathbf{O}_K)^*/\{\pm 1\}) , \pm|t|= h}g_{[t]}(\tau)  $$
 \textit{and} $ d=\displaystyle\frac{12}{\gcd(12,p+1)} \mbox{ divides }  \sum_{h}n^+_h$.}

In \cite[Chapter 5]{KL} Kubert and Lang studied the Cuspidal Divisor Class Group on the modular curve $ X(p^k) $. Since their description utilizes the parametrization of the set of cusps of $ X(p^k) $ by the elements of the quotient $ C_{ns}(p^k)/\{\pm 1\} $, it appears natural to develop and extend their techniques to non-split Cartan modular curves. Kubert and Lang proved the following:\\ 

\noindent
\textbf{Theorem \ref{CDCG}} \textit{If $ p\ge 5 $ consider $ R:= \mathbb{Z}[C_{ns}(p^k)/\{\pm 1\}] $ and let $ R_0 $ be the ideal of $ R $ consisting of elements of degree $0$.  
The Cuspidal Divisor Class Group $ \mathfrak{C}_{p^k} $ on $ X(p^k) $ is an $R-$module, more precisely there exists a Stickelberger element $ \theta \in \mathbb{Q}[C_{ns}(p^k)/\{\pm 1\}]$ such that,  under the identification of the group $ C_{ns}(p^k)/\{\pm 1\} $ with the set of cusps at level $ p^k $, the ideal $ R \cap R \theta  $ corresponds to the group of divisors of units in the modular function field $ F_{p^k} $ and:
$$ \mathfrak{C}_{p^k} \cong R_0 / R \cap R \theta. $$}
\noindent
In this theorem the authors exhibited an isomorphism reminding to a classical result in cyclotomic fields theory. Let $ J $ be a fractional ideal of $ \mathbb{Q}(\zeta_m) $ and $ G= $\mbox{Gal}$( \mathbb{Q}(\zeta_m)/\mathbb{Q}) \cong (\mathbb{Z}/m\mathbb{Z})^*$. Consider $\mathbb{Z}[G]$ acting on the ideals and ideal classes in the natural way: if $ x= \sum_{\sigma}x_{\sigma} \sigma  $ then $ J^x := \prod_{\sigma}(J^{\sigma})^{x_{\sigma}}.$ We have the following result:\\
\\ \textbf{Stickelberger's Theorem} \cite[pag. 333]{Washington} \textit{Define the Stickelberger element: 
$$ \theta = \sum_{a \scriptsize{\mbox{ mod }} m, (a,m)=1}\displaystyle \left\langle  \frac{a}{m}  \right\rangle \sigma_a^{-1} \in \mathbb{Q}[G].$$
The Stickelberger ideal $ \mathbb{Z}[G]\cap\theta\mathbb{Z}[G] $ annihilates the ideal class group of $ \mathbb{Q}(\zeta_m)$.}\\
 
\noindent Along these lines, the main result can be summarized as follows:\\ 
 
\noindent
\textbf{Main Theorem \ref{main}}\textit{ Consider $ p \ge 5 $, $ H:=(\mathbb{Z}/p^k\mathbb{Z})^*/\{\pm 1\}$ and $ w $ a generator of $ H $. There exists a Stickelberger element 
$$ \theta := \displaystyle\frac{p^k}{2} \sum_{i=1}^{\frac{p-1}{2}p^{k-1}} {\displaystyle\sum_{ \pm|s|=w^i, s \in ((\mathbf{O}_K / p^k \mathbf{O}_K)^*/\{\pm 1\}} B_2 \left( \left\langle \frac{\frac{1}{2}(s+\overline{s}) }{p^k} \right\rangle \right) } w^{-i} \in \mathbb{Q}[H] $$ such that, under the identification of the group $ H $ with the set of cusps of $ X^+_{ns}(p^k) $, the ideal $ \mathbb{Z}[H]\theta \cap \mathbb{Z}[H] $ represents the group of divisors of units of $ X^+_{ns}(p^k)$. The Cuspidal Divisor Class Group on $ X^+_{ns}(p^k) $ is a module over $\mathbb{Z}[H]$ and, more precisely, we have: } 
$$ \mathfrak{C}^+_{ns}(p^k) \cong \mathbb{Z}_0[H] / (\mathbb{Z}[H]\theta \cap \mathbb{Z}[H]).
  $$
        
From the previous statement we will show another result which has a counterpart in cyclotomic field theory. \\  

\noindent
\textbf{Theorem \ref{Cardinalità}} \textit{For any character $ \chi $ of    $C_{ns}(p^k)/\{\pm I\}$  (identified  with an even character of $ C_{ns}(p^k) $), we let:}
$$ B_{2,\chi} =  \sum_{\alpha \in  C_{ns}(p^k)/\{\pm I\}  } B_2 \left( \left\langle \frac{T(\alpha)}{p^k} \right\rangle \right) \chi(\alpha)  $$ 
\textit{where $ B_2(t) = t^2 - t + \frac{1}{6} $ is the second Bernoulli polynomial and $T$ is a certain $ (\mathbb{Z}/p^k\mathbb{Z})$-linear map. Then we have:}
 $$ |\mathfrak{C}^+_{ns}(p^k)| = 
\displaystyle 24\frac{ \displaystyle\prod_{}{\frac{p^k}{2}B_{2,\chi}}}{\gcd(12,p+1)(p-1)p^{k-1}}    $$  
\textit{where the product runs over all nontrivial characters $ \chi $ of $ C_{ns}(p^k)/{\pm I} $ such that $ \chi(M)=1 $ for every $ M \in C_{ns}(p^k) $
 with $ \det M = \pm 1 $.\\
In particular, for $ k=1 $ let $ \omega$ be a generator of the character group of $ C_{ns}(p) $ and $ v $ a generator of $ \mathbb{F}_{p^2}^* $. Then:}
$$ |\mathfrak{C}^+_{ns}(p)| = \displaystyle \frac{24}{(p-1)\gcd(12,p+1)}\prod_{j=1}^{\frac{p-3}{2}}\frac{p}{2}B_{2,\omega^{(2p+2)j}} = $$  
  $$ = \displaystyle\frac{ 576 \left| \det\left[\displaystyle\frac{p}{2}\left(\displaystyle\sum_{l=0}^{p}B_2\left(  \left\langle \frac{\frac{1}{2}\mbox{Tr}(v^{i-j+l\frac{p-1}{2}}) }{p} \right\rangle \right) - \frac{p+1}{6} \right) \right]_{1\le i,j \le \frac{p-1}{2}} \right|}{(p-1)^2 p (p+1)\gcd(12,p+1)}. $$    
 
This theorem could be considered analogous to the relative class number formula \cite[Theorem 4.17]{Washington}:
$$ h^-_m = Q w \prod_{\chi \mbox{ odd}} -\frac{1}{2}B_{1,\chi} $$
where $ Q=1 $ if $ m $ is a prime power and $ Q=2 $ otherwise, $ w $ is the number of roots of unity in $ \mathbb{Q}(\zeta_m) $ and we encounter the classical generalized Bernoulli numbers:
$$ B_{1,\chi} := \sum_{a=1}^{m}\chi(a)B_1\left( \frac{a}{m} \right)= \frac{1}{m} \sum_{a=1}^{m} \chi(a)a  \mbox{  for } \chi\not=1.$$

In the last section we will explicitly calculate $ |\mathfrak{C}^+_{ns}(p)| $  for some $ p \le 101 $. Consider the isogeny (cfr.\cite[Paragraph 6.6]{Diamond:mf}):$$ {J_0^+}^{new}(p^2) \longrightarrow \mathop{\bigoplus_{f}} A'_{p,f}  $$   
where the sum is taken over the equivalence classes of newforms $ f\in  S_2(\Gamma^+_0(p^2))$. From Theorems \ref{x1}, \ref{x2} and \ref{x3} we deduce that: 
 $$ 
|\mathfrak{C}^+_{ns}(p)| \mbox{ divides } \prod_f \mathop \textbf{\mbox{ gcd }}_{ \begin{scriptsize} \begin{array}{c} q \mbox{ prime}, q \nmid |\mathfrak{C}^+_{ns}(p)|,   \\  q \equiv \pm 1 \mbox{ mod }p \end{array} \end{scriptsize} } |A'_{p,f}(\mathbb{F}_{q})|.  $$  
Using the modular form database of W.Stein, we will find out that
for $ p \le 31 $:  
$$|\mathfrak{C}^+_{ns}(p)| = \prod_f \mathop
\textbf{\mbox{ gcd }}_{ \begin{scriptsize} \begin{array}{c} q < 500 \mbox{ prime},   \\  q \equiv \pm 1 \mbox{ mod }p \end{array} \end{scriptsize} } |A'_{p,f}(\mathbb{F}_{q})|.  $$

\section{Galois groups of modular function fields}
Following \cite[Chapter 1]{Silverman}, let
$  \mathbb{H} = \{x + iy \;| y > 0; x, y \in \mathbb{R} \} $ 
be the upper-half plane and \textit{n} a positive integer. 
The principal congruence subgroup of level \textit{n} 
is the subgroup of  
$ SL_2(\mathbb{Z}) $
 defined as follows: 
\begin{center}
$ \Gamma(n) := \left\{\begin{pmatrix} a&b\\c&d \end{pmatrix} \in SL_2({\mathbb{Z}}) : a\equiv d\equiv 1,~b\equiv c\equiv 0\mod n\right\}. $
\end{center}  
Then the quotient space 
$ \Gamma(n) \textbackslash \mathbb{H} $
is complex analytically isomorphic to an affine curve 
$ Y(n) $
that can be compactified by considering
$ \mathbb{H}^*:= \mathbb{H} \cup \mathbb{Q} \cup \{\infty\}  $
and by taking the extended quotient:
\begin{center}
$X(n)= \Gamma(n)\textbackslash \mathbb{H}^* = Y(n)\cup\Gamma(n)\textbackslash (\mathbb{Q}\cup\{\infty\}).$
\end{center}
The points $ \Gamma(n) \tau $ in $ \Gamma(n)\textbackslash (\mathbb{Q}\cup\{\infty\}) $ are called the cusps of $ \Gamma(n) $
and can be described by the fractions \textit{s=$\frac{a}{c}$} with $0 \leq a \leq n-1 $, $0 \leq c \leq n-1 $ and gcd(\textit{a,c})=1.
As a consequence, it is not difficult to infer that $ X(n) $ has 
$\displaystyle \frac{1}{2} n^2 \displaystyle\prod_{p | n} \Big(1- \frac{1}{p^2}\Big) $
cusps.
\\
Let $F_{n,\mathbb{C}}$ the field of modular functions of level $n$. A classical result states that $F_{1,\mathbb{C}} = \mathbb{C}(j)$ where \textit{j} is the Klein's \textit{j}-invariant.
We shall now find generators for $F_{n,\mathbb{C}}$. Consider:
\begin{center}
$ f_0(w;\tau) = -2^7 3^5 \displaystyle\frac{g_2(\tau)g_3(\tau)}{\Delta(\tau)}\wp(w;\tau,1),$
\end{center} 
where $ \Delta$ is the modular discriminant, $ \wp$ is the Weierstrass elliptic function, $ \tau \in \mathbb{H} $, $w \in \mathbb{C}$ and $ g_2 = 60G_4 $ and $ g_3 = 140G_6 $ are constant multiples of the Eisenstein series: 
  $$      G_{2k}(\tau) = \sum_{{ (m,n)\in\mathbf{Z}^2\backslash(0,0)}} \frac{1}{(m+n\tau )^{2k}}.
  $$   
For $ r,s \in \mathbb{Z} $ and not both divisible by \textit{n} we define $ f_{r,s}= f_{0}(\frac{r\tau +s}{n}; \tau)  $. 
Whereas the Weierestrass   $ \wp$-function is elliptic with respect to the lattice $[\tau,1]$ it follows that $ f_{r,s} $ depends only on the residue of $r,s$ mod $n$. Thus, it is convenient to use a notation emphasizing this property. 
If $ a=(a_1,a_2) \in \mathbb{Q}^2$ but  $a \not\in  \mathbb{Z}^2$ we call the functions $ f_a(\tau)=f_0(a_1\tau+a_2;\tau)$   the Fricke functions. They depend only on the residue class of $a $ mod $ \mathbb{Z}^2 $. 

\begin{thm}
We have:
$$  \mbox{ Gal}( F_{n,\mathbb{C}}, F_{1,\mathbb{C}})\cong SL_2(\mathbb{Z}/n\mathbb{Z})/ \{\pm I \}.  $$
\end{thm}
\begin{proof}
There is a surjective homeomorphism (see \cite[pag.279]{Diamond:mf} and \cite[pag.65]{Lang:ef}):
\begin{center}
$ \theta: SL_2(\mathbb{Z}) \longrightarrow $ Aut $(\mathbb{C}(X(n))), $
\end{center} 
\begin{center} 
$ \gamma \longmapsto   $
$ (f \longmapsto f^{(\theta(\gamma))} = f \circ \gamma). $
\end{center}
From Ker$(\theta) = \pm \Gamma(n) $ and the relations $ f_a (\gamma(\tau)) = f_{a\gamma}(\tau) $ it follows easily that Gal$( F_{n,\mathbb{C}}, F_{1,\mathbb{C}})\cong \Gamma(1)/\pm \Gamma(n)  \cong SL_2(\mathbb{Z}/n\mathbb{Z})/ \{\pm I \} $.
\end{proof}
We say that a modular form in $ F_{n,\mathbb{C}} $ is defined over a field if all the coefficients of its $q$-expantion lie in that field and analogously for every $\mbox{Gal}( F_{n,\mathbb{C}}, F_{1,\mathbb{C}})$-conjugate of the form. Let:
\begin{center}
\begin{flushleft}
$ F_n= $ 
function field on $ X(n) $ consisting of those functions which are defined over the \textit{n}-th cyclotomic field 
$ \mathbb{Q}_n = \mathbb{Q}(\zeta_n) $.
\end{flushleft}
\end{center} 

\begin{thm}\label{fricke}
The field $ F_n $ has the following properties:\\
(1) $ F_n $ is a Galois extension of $ F_1 = \mathbb{Q}(j) $.\\
(2) $ F_n = \mathbb{Q}(j,f_{r,s}) 
 _{all (r,s)\in \frac{1}{n}\mathbb{Z}^2 \smallsetminus \mathbb{Z}^2 }              $. \\
(3) For every $ \gamma \in GL_2(\mathbb{Z}/n\mathbb{Z})$ the map   $ f_a \mapsto f_{a\gamma}$     gives an element of Gal$(F_n,\mathbb{Q}(j))$ which we write   $\theta(\gamma)  $. Then $ \gamma \mapsto \theta(\gamma) $ induces an isomorphism of $  GL_2(\mathbb{Z}/n\mathbb{Z})/{\pm I} $  to Gal$(F_n,\mathbb{Q}(j))$. The subgroup $ SL_2(\mathbb{Z}/n\mathbb{Z})/{\pm I} $ operates on a modular function by composition with the natural action of $ SL_2(\mathbb{Z})$ on the upper half-plane $\mathbb{H} $. \\
Furthermore the group of matrices $ \begin{pmatrix} 1&0\\0&d \end{pmatrix}   $ operates on $ F_n $ as follows: \\
for $d \in (\mathbb{Z}/n\mathbb{Z})^*  $ consider the automorphism $ \sigma_d $ of $ \mathbb{Q}_n $ such that $ \sigma_d(\zeta_n)= \zeta_n^d $. Then $ \sigma_d $ extends to $ F_n $ by operating on the coefficients of the power series expansions:
\begin{center}
$ \sigma_d (\sum{a_i q^{i/n}})=  \sum{\sigma_d(a_i) q^{i/n}} $ with $ q=e^{2\pi i \tau}$.  \\
\end{center} 
If $ (r,s) \in \frac{1}{n}\mathbb{Z}^2 \setminus \mathbb{Z}^2 $, we have: $\sigma_d(f_{r,s}(\tau))=f_{r,sd}(\tau).  $
 
\end{thm}
\begin{proof}

 \cite[Theorem 6.6]{Shimura:af}
 \end{proof}
 
\section{Modular Units and Manin-Drinfeld Theorem} 
In this paper we will focus our attention on the modular units of $ X(n) $. In other words, the invertible elements of the integral closure of $ \mathbb{Q}[j] $ in $ F_n $. The only pole of \textit{j}$ (\tau) $ is at infinity. So, from the algebraic characterization of the integral closure as the intersection of all valuation subrings containing the given ring, the modular units in $ F_n $ are exactly the modular functions which have poles and zeros exclusively at the cusps of $ X(n) $.

Let $ \mathfrak{D}_n \simeq \bigoplus_{\footnotesize{\mbox{cusps}}} \mathbb{Z} $ be the free abelian group of rank $ {\frac{1}{2} n^2 \prod_{p | n} (1- \frac{1}{p^2})} $ generated by the cusps of $ X(n) $. Let  $ \mathfrak{D}_{n,0} $ be its subgroup consisting of elements of degree $ 0 $ and let $ \mathfrak{F}_n $ be the subgroup generated by the divisors of modular units in the modular function field $ F_n $. The quotient group:   
\begin{center}
$  \mathfrak{C}_n := \mathfrak{D}_{n,0} / \mathfrak{F}_n $
\end{center}
is called the Cuspidal Divisor Class Group on $ X(n) $. The previous definition generalizes \textit{mutatis mutandis} to every modular curve $ X_\Gamma $ where $ \Gamma $ is a modular subgroup. Manin and Drinfeld proved that:
\begin{thm}\label{mandrinf}
 If $ \Gamma $ is a congruence subgroup then all divisors of degree 0 whose support is a subset of the set of cusps of $ X_\Gamma $ have a multiple that is a principal divisor. In other word if $ x_1,x_2 \in X_\Gamma $ are cusps, then $ x_1 - x_2 $ has finite order in the jacobian variety $ Jac(X_\Gamma) $.   
\end{thm}
\begin{proof}
Let $ x_1,x_2$ two cusps in $ X_\Gamma $. Denote by $ \{x_1,x_2 \} \in (\Omega^1(X_\Gamma))^*$  the functional on the space of differential of the first kind given by:
$$  \{x_1, x_2 \}: \omega \mapsto \int_{x_1}^{x_2}\omega. $$  
  
\textit{A priori} we have $ \{x_1,x_2\} \in H_1(X_\Gamma, \mathbb{R}) $. Manin and Drinfeld showed that it lies in  $H_1(X_\Gamma, \mathbb{Q})$. Cf. \cite{Drinfeld}, \cite[Chapter IV]{Lang:mf} and \cite{Manin}.
\end{proof}

\section{Siegel Functions and Cuspidal Divisor Class Groups}

Let $ n=p^k $ with $ p \ge 5 $ prime. Following \cite{KL} we will give an explicit description of modular units of $ X(n) $ and its cuspidal divisor class group. \\
Let $L$ a lattice in $ \mathbb{C} $. Define the Weierstrass sigma function:
$$ \sigma_L (z) = z \displaystyle\prod_{ \begin{scriptsize} \begin{array}{c} \omega \in L \\  \omega \not=0 \end{array} \end{scriptsize} }{ \left( 1 - \frac{z}{\omega} \right)  e^{z/\omega + \frac{1}{2} (z/\omega)^2  }                 },  $$
\noindent which has simple zeros at all non-zero lattice points. Define:
$$  \zeta_L (z) = \frac{d}{dz} \log (\sigma_L(z)) = \displaystyle{\frac{1}{z} + \displaystyle\sum_{ \begin{scriptsize} \begin{array}{c} \omega \in L \\  \omega \not=0 \end{array} \end{scriptsize} }{\left[\frac{1}{z-\omega} +\frac{1}{\omega} + \frac{z}{\omega^2} \right]} }, $$
$$ \wp_L(z)= -\zeta'_L(z) = \frac{1}{z^2} + 
\displaystyle\sum_{ \begin{scriptsize} \begin{array}{c} \omega \in L \\  \omega \not=0 \end{array} \end{scriptsize} }{\left[\frac{1}{(z-\omega)^2} - \frac{1}{\omega^2} \right]}. 
 $$
\noindent
If $ \omega \in L $, by virtue of the periodicity of $ \wp_L $ we obtain $ \frac{d}{dz} \zeta_L(z+\omega) = \frac{d}{dz}\zeta_L(z)$, whence follows the existence of a $ \mathbb{R}$-linear function $ \eta_L(z) $ such that:
$$ \zeta_L(z+\omega) = \zeta_L(z) + \eta_L(\omega).$$
For $ L=[\tau,1] $ (with $\tau \in \mathbb{H} $) and $ a=(a_1,a_2) \in \mathbb{Q}^2 \smallsetminus \mathbb{Z}^2 $ we define the Klein forms:
$$ \mathfrak{k}_a(\tau) = e^{-\eta_L(a_1\tau+a_2)}\sigma_L(a_1\tau+a_2).$$

Note that $ z=a_1\tau+a_2 \not\in L=[\tau,1] $ so we know directly from their definition that the Klein forms are holomorphic functions which have no zeros and poles on the upper half plane. 

When $ \Gamma $ is a congruence subgroup and $ k $ is an integer, we will say that a holomorphic function $ f(\tau) $ on $ \mathbb{H}$ is a nearly holomorphic modular form for $ \Gamma $ of weight $ k $ if:\\
(i) $ f(\gamma(\tau))= (r\tau +s)^k f(\tau) $ for all  $\gamma=\begin{pmatrix} p&q\\r&s \end{pmatrix}  \in \Gamma $;\\
(ii) $ f(\tau) $ is meromorphic at every cusp. 

 \begin{prop}\label{Kleinforms}
  Let $ a=(a_1,a_2) \in \mathbb{Q}^2 \smallsetminus \mathbb{Z}^2 $ and $ b=(b_1,b_2) \in \mathbb{Z}^2 $. The Klein Forms $ \mathfrak{k}_a(\tau) $ have the following properties: \\
  (1) $ \mathfrak{k}_{-a}(\tau) = - \mathfrak{k}_a(\tau); $\\
  (2) $ \mathfrak{k}_{a+b} = \epsilon(a,b)\mathfrak{k}_a(\tau) $ with $ \epsilon(a,b)=(-1)^{b_1 b_2 + b_1 + b_2} e^{- \pi i (b_1 a_2 - b_2 a_1)}; $\\
  (3) For every  $\gamma=\begin{pmatrix} p&q\\r&s \end{pmatrix}  \in SL_2(\mathbb{Z}) $ we have: 
  $$ \mathfrak{k}_a (\gamma(\tau)) = \mathfrak{k}_a \left(\frac{p\tau+q}{r\tau+s}\right) = \frac{\mathfrak{k}_{a\gamma}(\tau)}{r \tau + s} =
   \frac{\mathfrak{k}_{(a_1 p + a_2 r , a_1 q + a_2 s  )}(\tau)}{r \tau + s};  $$ 
  (4) If $ n \ge 2 $ and $ a \in \frac{1}{n}\mathbb{Z}^2 \setminus \mathbb{Z}^2 $  then $ \mathfrak{k}_a(\tau) $ is a nearly holomorphic modular form for $ \Gamma(2n^2) $ of weight -1. \\
  (5) Let  $ n \ge 3 $ odd and $ \{m_a\}_{a \in \frac{1}{n}\mathbb{Z}^2 \setminus \mathbb{Z}^2} $ a family of integers such that $ m_a \not= 0 $ occurs only for finitely many $a$. Then the product of Klein form:
  $$ \prod_{a \in \frac{1}{n}\mathbb{Z}^2 \setminus \mathbb{Z}^2} \mathfrak{k}_a^{m_a}(\tau)   $$
  is a nearly holomorphic modular form for $ \Gamma(n)$ of weight $ -\sum_a m_a $ if and only if:
  $$ \sum_a {m_a (n a_1)^2 } \equiv  \sum_a {m_a (n a_2)^2 } \equiv  \sum_a {m_a (n a_1)(n a_2) \equiv 0 \mod n }. $$ 
 
    \end{prop}
\begin{proof}
Property (2) is nothing more than a reformulation of the Legendre relation: $ \eta_{[\tau,1]}(1)\tau - \eta_{[\tau,1]}(\tau) = 2 \pi i $. Property (5) is discussed in \cite[Chapter 3, Paragraph 4]{KL}.\\
For more details see: \cite[Chapters 2 and 3]{KL} or \cite[Chapter 19]{Lang:ef}. 
\end{proof}

We are now ready to define the Siegel function:
$$ g_a(\tau) = \mathfrak{k}_a(\tau)\Delta(\tau)^{1/12},$$

\noindent
where $ \Delta(\tau)$   is the square of the Dedekind eta funtion $ \eta(\tau) $ (not to be mistaken for the aforementioned $\eta_L(\tau) $):

\begin{center}
$ \eta(\tau)^2 = 2 \pi i  q^{1/12}\prod_{k=1}^{\infty}(1-q^n)^2 $ with $ q=e^{2\pi i \tau} .$
\end{center}

\begin{prop}\label{sigfun} 
The set of functions $ \{h_a(\tau)=g_a(\tau)^{12n} \} _{a \in \frac{1}{n}\mathbb{Z}^2 \setminus \mathbb{Z}^2} $ constitute a Fricke family. Just like the Fricke functions $ f_a(\tau) $ of Theorem \ref{fricke} we have: $ h_a(\tau) \in F_n $, for every $ \gamma \in SL_2({\mathbb{Z}}) $ we have $ h_a(\gamma(\tau))=h_{a\gamma}(\tau) $ and in addition if $ \sigma_d \in Gal(\mathbb{Q}_n, \mathbb{Q}) $ then $ \sigma_d (h_{a_1, a_2}(\tau)) = h_{a_1, d a_2}(\tau) $. In other words, the Siegel functions, raised to the appropriate power, are permuted by the elements of the Galois Group Gal$( F_n,\mathbb{Q}(j) )$.

\end{prop}
\begin{proof}
 \cite[Chapter 2]{KL} or \cite{Siegel}.
\end{proof} 

\begin{thm}\label{unitàmodulari}
Assume that $ n=p^k $ for $ p \not=2,3.$ Then the units in $ F_n $ (modulo constants) consist of the power products:
$$ \prod_{a \in \frac{1}{n}\mathbb{Z}^2 \setminus \mathbb{Z}^2} g_a^{m_a}(\tau)   $$
 with:
$$ \sum_a {m_a (n a_1)^2 } \equiv  \sum_a {m_a (n a_2)^2 } \equiv  \sum_a {m_a (n a_1)(n a_2) \equiv 0 \mod n } $$ 
and $$ \sum_a{m_a} \equiv 0 \mod 12. $$
\noindent 
In addition, if $ k \ge 2 $ it is not restrictive to consider power products of Siegel functions $g_a$ with primitive index $a=(a_1,a_2)$, namely such that $ p^{k-1}a \not\in \mathbb{Z}^2  $.

\end{thm}

\begin{proof}
See \cite{Siegelgenerator}, \cite[Theorem 3.2, Chapter 2]{KL}, \cite[Theorem 5.2, Chapter 3]{KL} and \cite[Theorem 1.1, Chapter 4]{KL} . The last assertion is a consequence of the distribution relations discussed in \cite[pp. 17-23]{KL}.
\end{proof}

Following \cite{KL} it will be useful to decompose  Gal($ F_{p^{k}},\mathbb{Q}(j) $). Let $ \mathfrak{o}_p $ the ring of integers in the unramified quadratic extension of the $p$-adic field $ \mathbb{Q}_p $. The group of units $ \mathfrak{o}_p^* $ acts on $ \mathfrak{o}_p $ by multiplication and after choosing a basis of  $ \mathfrak{o}_p $ over the $p$-adic ring $ \mathbb{Z}_p $, we obtain an embedding:
$$ \mathfrak{o}_p^* \longrightarrow GL_2(\mathbb{Z}_p).$$

\noindent
We call the image in $ GL_2(\mathbb{Z}_p) $ the Cartan Group at the prime $ p $ and indicate it by $C_p $. It is worth noting that the elements of $  \mathfrak{o}_p^* $, written in terms of a basis of $ \mathfrak{o}_p $ over $ \mathbb{Z}_p $, are characterized by the fact that at least one of the two coefficients is a unit. 

Consider now $ GL_2(\mathbb{Z}_p) $ as operating on $\mathbb{Z}_p^2 $ on the left and denote by $ G_{p,\infty} $ the isotropy group of $ \begin{pmatrix}
1 \\ 0
\end{pmatrix} $. Obviously we have:

\begin{center}
$  G_{p,\infty} = \left\{  \begin{pmatrix} 1&b\\0&d \end{pmatrix} b  \in \mathbb{Z}_p , d \in \mathbb{Z}_p^* \right\}. $ 

\end{center}

 Since $ C_p $ operates simply transitively on the set of primitive elements (that is: vectors whose coordinates are not both divisibile by $ p $) we have the following decomposition:
 $$ GL_2(\mathbb{Z}_p)/\{\pm I\} = ( C_p/\{\pm I\} )   G_{p,\infty}.  $$
\noindent
For each integer $ k $ we define the reduction of the Cartan Group $ C_p \mod p^k$:
$$ C(p^k)= C_p / p^k C_p $$
\noindent
and let $ G_\infty(p^k) $ the reduction of $ G_{p,\infty} \mod p^k $:

\begin{center}
$  G_\infty(p^k) = \left\{  \begin{pmatrix} 1&b\\0&d \end{pmatrix} b  \in \mathbb{Z}/p^k\mathbb{Z} , d \in  (\mathbb{Z}/p^k\mathbb{Z})^*  \right\}.$ 

\end{center}
\noindent
We can now reformulate the previous decomposition as follows: 

\begin{center}
Gal($ F_{p^k} $,$ \mathbb{Q}(j) $) $   \simeq GL_2( \mathbb{Z}/p^k\mathbb{Z})/\{\pm I\} = ( C(p^k)/\{\pm I\} )   G_{\infty}(p^k).   $  

\end{center}

\noindent
The embedding:
$$ F_{p^k} \hookrightarrow \mathbb{Q}(\zeta_{p^k})((q^{1/{p^k}})) $$
enables us to mesaure  for each modular function $ f(\tau) \in F_{p^k} $ its order at $ \Gamma(p^k) \infty $ in term of the local parameter $ q^{1/{p^k}} $. 

\begin{prop}
If $ a \in \frac{1}{p^k} \mathbb{Z}^2 \setminus \mathbb{Z}^2 $, the $q$-expansion of the Siegel functions shows that:
\begin{center} 

ord$_\infty (g_a(\tau))^{12p^k} = 6p^{2k} B_2(\langle a_1 \rangle)$

\end{center}

\noindent
where $ B_2(X) = X^2 -X + \frac{1}{6} $ is the second Bernoulli polynomial and $ \langle X \rangle $ is the fractional part of $ X $. 

\end{prop}
\begin{proof}
\cite[Chapter 19]{Lang:ef}.
\end{proof}

 For every automorphism $ \sigma \in $   Gal($ F_{p^{k}},\mathbb{Q}(j) $) and each $h(\tau) \in F_{p^k}$ we have the prime $ \sigma^{-1}(\infty) $ which is such that:

\begin{center}
ord$_{\sigma^{-1}(\infty)} (h(\tau)) = $ ord$_\infty \sigma (h(\tau)) $ 
\end{center}

\noindent
and if $  \sigma \in G_\infty(p^k)  $:

\begin{center}
ord$_\infty (h(\tau)) = $ ord$_\infty \sigma (h(\tau)), $
\end{center}

\noindent 
so we may identify the cusps of $ X(p^k) $ with the elements of the Cartan Group (viewing it as a subgroup of  Gal($ F_{p^{k}},\mathbb{Q}(j) $)). From now on, we will indicate the cusp $ \sigma^{-1}(\infty) $ simply by $ \sigma^{-1} $. \\
We may also index the primitive Siegel function by elements of the Cartan Group. Following \cite{KL}, if $ \alpha \in  C(p^k)/\{\pm I\} $ we put:

\begin{center}  

$ g_\alpha = g_{e_1 \alpha } $ where $ e_1=(\frac{1}{p^k},0). $

\end{center} 
\noindent 
It should be noted that $ g_\alpha $ is defined up to a root of unity (this follows from Proposition \ref{Kleinforms}, second claim). Nonetheless, $ g_\alpha ^{12p^k} $ is univocally defined as well as its divisor:

\begin{prop}\label{divisori} We have:
$$ \mbox{div }  g_\alpha^{12p^k} = 6p^{2k} \sum_{\beta \in  C(p^k)/\{\pm I\}  } B_2 \left( \left\langle \frac{T(\alpha\beta^{-1})}{p^k} \right\rangle \right) \beta  $$ 

\noindent 
where the map T on $ 2 \times 2 $ matrices is defined as follows:
$$ T:\begin{pmatrix} a&b\\c&d \end{pmatrix} \mapsto a . $$

\end{prop}

\begin{proof}
See \cite[Paragraph 5.1]{KL}
\end{proof}

The first part of \cite{KL} culminates with the theorem below. The computation of the order of the cuspidal divisor class group on $ X(p^k) $ could be considered analogous to that in the study of cyclotomic fields: instead of the generalized Bernoulli numbers $ B_{1,\chi} $ encountered in the latter case, in the former we will define the second generalized Bernoulli numbers $ B_{2,\chi} $.  
 
\begin{thm}\label{CDCG}
Let $ p $ a prime $ \ge 5 $. Let $ R:= \mathbb{Z}[C(p^k)/\{\pm 1\}] $ and $ R_0 $ the ideal of $ R $ consisting of elements of degree $ 0$. The Cuspidal Divisor Class Group $ \mathfrak{C}_{p^k} $ is an $R-$module, more precisely there exists a Stickelberger element
$$ \theta = \displaystyle \frac{p^k}{2}\sum_{\beta \in C(p^k)/\{\pm 1\} } B_2 \left( \left\langle\frac{T(\beta)}{p^k} \right\rangle \right) \beta^{-1}  \in \mathbb{Q}[C(p^k)/\{\pm 1\}]$$ such that:
$$ \mathfrak{C}_{p^k} \cong R_0 / R \cap R \theta. $$

 For any character $ \chi $ of    $C(p^k)/\{\pm I\}$  (identified  with an even character of $ C(p^k) $) we let:
$$ B_{2,\chi} =  \sum_{\alpha \in  C(p^k)/\{\pm I\}  } B_2 \left( \left\langle \frac{T(\alpha)}{p^k} \right\rangle \right) \chi(\alpha).  $$ 
\noindent
The order of the cuspidal divisor class group on $ X(p^k) $ is:
$$ | \mathfrak{C}_{p^k}  | =  \frac{12 p^{3k}}{|C(p^k) |} \prod_{\chi \not= 1} \frac{p^k}{2} B_{2,\chi}. $$

\end{thm} 
 
 \begin{proof}
 \cite[Chapter 5]{KL}. 
 \end{proof}

\section{Non-split Cartan Groups} 

Following \cite{BB} or \cite[pag. 194]{SerreMordel}, let $n $ a positive integer and let $ A $ be a finite free commutative algebra of rank $ 2 $ over $ \mathbb{Z}/n\mathbb{Z}$ with unit discriminant. Fixing a basis for $ A $ we can use the action of $ A^* $ on $ A  $ to embed $ A^* $ in $ GL_2(\mathbb{Z}/n\mathbb{Z}) $. If for every prime $ p|n $ the $ \mathbb{F}_p $ algebra $ A/pA $ is isomorphic to $ \mathbb{F}_{p^2} $, the image of $ A^* $ just now described is called a non-split Cartan subgroup of $ GL_2(\mathbb{Z}/n\mathbb{Z}) $.  
Therefore, such a group $ G $ has the property that for every prime $p$ dividing $n$ the reduction of $G\mod p $ is isomorphic to $ \mathbb{F}_{p^2}^* $. All the non-split Cartan subgroups of $ GL_2({\mathbb{Z}/n\mathbb{Z}}) $ are conjugate and so are their normalizers.  

In this paper we are interested in the case $ n=p^k$ and $ p \not= 2,3$. The cases $p=2$ and $ p=3 $ are essentially equal but require more cumbersome calculations (see \cite[Theorem 5.3, Chapter 3]{KL} and \cite[Theorem 1.3, Chapter 4]{KL}).  Choose a squarefree integer $ \epsilon \equiv 3 \mod 4 $ and such that its reduction modulo $ p$ is a quadratic non-residue. If $ p \equiv 3 \mod 4 $, a canonical choice could be $ \epsilon = -1 $.
Let $ K= \mathbb{Q}(\sqrt{\epsilon}) $ and $ \mathbf{O}_K = \mathbb{Z}[\sqrt{\epsilon}]$  its ring of integers. After choosing a basis for $ \mathbf{O}_K $ over $ \mathbb{Z} $ we can represent any element of  $ (\mathbf{O}_K / p^k \mathbf{O}_K)^*  $ with its corresponding multiplication matrix in   $ GL_2({\mathbb{Z}/p^k\mathbb{Z}}) $ with respect to the chosen basis. This embedding produces a non-split Cartan subgroup of $ GL_2({\mathbb{Z}/p^k\mathbb{Z}}) $ and we will denote it by $ C_{ns}(p^k) $. Notice that such a group is isomorphic to the already introduced $ C(p^k) $. 

To describe the normalizer $ C_{ns}^+(p^k) $ of $ C_{ns}(p^k) $ in $ GL_2({\mathbb{Z}/p^k\mathbb{Z}}) $ it will suffice to consider the following group automorphism induced by conjugation by a fixed $c \in  C_{ns}^+(p^k) $:
$$ \phi_c :   C_{ns}(p^k)  \longrightarrow   C_{ns}(p^k)   $$
$$  x \longmapsto \phi_c(x) = cxc^{-1}.  $$
\noindent
The group automorphism $ \phi_c $ extends to a ring automorphism of $ (\mathbf{O}_K / p^k \mathbf{O}_K) \cong (\mathbb{Z}/p^k\mathbb{Z})[\sqrt{\epsilon}] $ so if $ \phi_c $ is not the trivial automorphism we necessarily have $ \phi_c(\sqrt{\epsilon})=-\sqrt{\epsilon} $.

\begin{prop}\label{strutturanonsplit} If $ p \not= 2 $ we have the following isomorphism:
$$ C_{ns}(p^k) \simeq \mathbb{Z}/p^{k-1}\mathbb{Z} \times  \mathbb{Z}/p^{k-1}\mathbb{Z} \times \mathbb{Z}/(p^2-1)\mathbb{Z},  $$
$$ C_{ns}^+(p^k) \simeq (\mathbb{Z}/p^{k-1}\mathbb{Z} \times  \mathbb{Z}/p^{k-1}\mathbb{Z} \times \mathbb{Z}/(p^2-1)\mathbb{Z}) \rtimes_{\phi} \mathbb{Z}/2\mathbb{Z}.  $$
\end{prop}  

\begin{proof}
Let $ a_1 + \sqrt{\epsilon}a_2 \in  (\mathbf{O}_K / p^k \mathbf{O}_K)$: it is invertible if and only if $ (a_1,a_2)$ is primitive or in other words $ p $ does not divide both $ a_1 $ and $ a_2 $ so we have $ |C_{ns}(p^k)| = p^{2k-2}(p^2-1) $. 
Consider the reduction$\mod p $:
$$ C_{ns}(p^k) \longrightarrow \mathbb{F}_{p^2}^*  $$
$$  a_1 + \sqrt{\epsilon}a_2 \longmapsto  \overbar{a_1} + \sqrt{\epsilon}\overbar{a_2}. $$
 The map is surjective and let $B$ its kernel:
 $$ B:=\{x \in (\mathbf{O}_K / p^k \mathbf{O}_K)^*  \mbox{ such that } x \equiv 1\mod p \}. $$
$ |B|=p^{2k-2} $: it remains to check that $ B \simeq \mathbb{Z}/p^{k-1}\mathbb{Z} \times  \mathbb{Z}/p^{k-1}\mathbb{Z} $.
Let $ k\ge2 $ and $ p \not=2 $. 
First, we check that for all $ x \in \mathbf{O}_K $ we have $ (1+xp)^{p^{k-2}} \equiv 1+xp^{k-1} \mod p^k $. In case $ k=2 $ there is nothing to prove. We proceed by induction on $ k $: suppose the claim is true for some $ k\ge2 $. We have:
 $$  (1+xp)^{p^{k-2}} = 1+xp^{k-1} + yp^k, $$ 
$$ (1+xp)^{p^{k-1}} = \sum_{j=0}^p {p \choose j}{(1+xp^{k-1})}^{p-j}({yp^k})^j  \equiv (1+xp^{k-1})^p \mod p^{k+1},  $$  
 $$  (1+xp^{k-1})^p = \sum_{j=0}^p {p \choose j}(xp^{k-1})^j \equiv 1 + xp^k \mod p^{k+1}.  $$

\noindent
In conclusion: $ (1+xp)^{p^{k-1}} \equiv 1 + xp^k \mod p^{k+1}  $. From the previous claim follows that if $ h\le k-1$ is such that $ x \in p^h\mathbf{O}_K \setminus p^{h+1}\mathbf{O}_K $ then the reduction of $ 1+xp $ in $ B $ has order $ p^{k-1-h} $. So $B $ has $ p^{2k-2}-p^{2k-4}$ elements of order $ p^{k-1} $ and the proposition is proved. The second isomorphism follows immediately.      
   
\end{proof} 

We present now the modular curves $ X_{ns}(n) $ and  $ X_{ns}^+(n) $ associated to the subgroups $ C_{ns}(n) $ and  $ C_{ns}^+(n) $. First of all, $ Y(n) $ (the non-cuspidal points of $ X(n) $) are isomorphism classes of pairs $(E,(P,Q)) $ where $ E $ is a complex elliptic curve and $ (P,Q) $ constitute a $ \mathbb{Z}/n\mathbb{Z}$-basis of the $n$-torsion subgroup $ E[n]$ with $ e_n(P,Q)=e^{2\pi i / n}$ where $ e_n $ is the Weil pairing discussed in details in \cite[Chapter 7]{Diamond:mf}. 
By definition, two pairs $(E,(P,Q)) $ and  $(E',(P',Q')) $ are considered equivalent in $ Y(n) $ if and only if there exists an isomorphism between $ E $ and $ E'$ taking $ P $ to $ P' $ and $ Q $ to $ Q'$. Notice that the definition is well-posed since the Weil pairing is invariant under isomorphism, i.e. if $ f: E \rightarrow E' $ is an isomorphism of elliptic curves and $ e'_n $ is the Weil pairing on $ E' $ we have: 
$$ e'_n(f(P),f(Q))=e_n(P,Q).   $$
\noindent
Since  $ GL_2(\mathbb{Z}/n\mathbb{Z}) $ acts on $ E[n] $ and since for every $ \gamma \in GL_2(\mathbb{Z}/n\mathbb{Z}) $ we have $ e_n(\gamma(P,Q)) = e_n(P,Q)^{\det \gamma} $,
the group $ SL_2(\mathbb{Z}/n\mathbb{Z}) $ acts on $ Y(n) $ on the right in the following way: 
$$ \begin{pmatrix} a&b\\c&d \end{pmatrix} \cdot(E,(P,Q)) = (E,(aP+cQ,bP+dQ)).  $$
Define:  
 $$ C'_{ns}(n) :=  C_{ns}(n) \cap SL_2(\mathbb{Z}/n\mathbb{Z}), $$  
$$ C'^+_{ns}(n) :=  C^+_{ns}(n) \cap SL_2(\mathbb{Z}/n\mathbb{Z}), $$ 
$$ \Gamma_{ns}(n) := \{M \in SL_2({\mathbb{Z}}) \mbox{ such that }  M \equiv M'\mbox{ mod } n\ \mbox{for some }  M' \in  C'_{ns}(n)\},   $$
$$ \Gamma^+_{ns}(n) := \{M \in SL_2({\mathbb{Z}}) \mbox{ such that }  M \equiv M'\mbox{ mod } n\ \mbox{for some }  M' \in  C'^+_{ns}(n)\}.   $$

\noindent A possible explicit description for these groups is: 
$$ C_{ns}(p^k)= \left\{M_s=\begin{pmatrix} a&b\\\epsilon b&a \end{pmatrix} \in GL_2({\mathbb{Z}/p^k \mathbb{Z}})  \mbox{ with } s=a+\sqrt{\epsilon}b \in (\mathbf{O}_K / p^k \mathbf{O}_K)^* \right\}, $$
$$ C^+_{ns}(p^k) = \left\langle {\begin{pmatrix} a&b\\\epsilon b&a \end{pmatrix}} \in C_{ns}(p^k) \mbox{ , } {C=\begin{pmatrix} 1&0\\ 0&-1 \end{pmatrix}}    \right\rangle. $$

\noindent If $ s  \in (\mathbf{O}_K / p^k \mathbf{O}_K)^* $ we define $ |s|:= s \bar{s} \in (\mathbb{Z} / p^k \mathbb{Z})^*$ where $ \bar{s} $ is the conjugate of $ s $. So we have:
$$ C'_{ns}(p^k)= \left\{M_s=\begin{pmatrix} a&b\\\epsilon b&a \end{pmatrix} \in C_{ns}(p^k)  \mbox{ such that } |s|=|a+\sqrt{\epsilon}b|=1 \mbox{ mod } p^k \right\},  $$
$$  C'^+_{ns}(p^k) =  C'_{ns}(p^k) \cup \left\{M_s C=\begin{pmatrix} a&-b\\\epsilon b&-a \end{pmatrix} \mbox{ with } |s|=|a+\sqrt{\epsilon}b|=-1 \mbox{ mod } p^k \right\}.   $$

Points in $ Y_{ns}(n) $ are nothing but orbits of $Y(n)$ under the action of $ C'_{ns}(n) $ and similarly for $ Y_{ns}^+(n) $ and $ C'^+_{ns}(n) $. The above-mentioned action extends uniquely to $ X(n) $. The quotients $ X_{ns}(n) $ and $ X^+_{ns}(n) $ are  isomorphic as Riemann surfaces to $ \mathbb{H}^*/\Gamma_{ns}(n) $ and $ \mathbb{H}^*/\Gamma^+_{ns}(n) $ respectively.  

Using the identification of the cusps of $ X(p^k) $ with the elements of $ C(p^k)/\{\pm I\}$ explained in the previous section we obtain a shorter proof of the first claim of \cite[Proposition 7.10]{BB}:    

\begin{prop}\label{numerocuspidi} We identify the cusps of $ X_{ns}(p^k) $ with $ (\mathbb{Z}/p^k\mathbb{Z})^* $ and the cusps of $ X^+_{ns}(p^k) $ with $ H=(\mathbb{Z}/p^k\mathbb{Z})^*/\{\pm 1\} $. So $ X_{ns}(p^k) $ has $ p^{k-1}(p-1) $ cusps and $ X^+_{ns}(p^k) $ has $ p^{k-1}\frac{p-1}{2} $ cusps.  
\end{prop}

\begin{proof}
We identify the cusps of $ X(p^k) $ with the elements of $ C(p^k)/\{\pm I\} \cong C_{ns}(p^k)/\{\pm I\} \cong (\mathbf{O}_K / p^k \mathbf{O}_K)^*/\{\pm 1\}  $. Bearing this in mind, it is clear that $ \pm M_r, \pm M_{r'} \in C_{ns}(p^k)/\{\pm I\} $ represent the same cusp in $ X_{ns}(p^k) $ if and only if there exists $ s \in  (\mathbf{O}_K / p^k \mathbf{O}_K)^* $ with $ |s|=1 $ such that $ \pm r=\pm s r' $. But this is equivalent to say that $ |r|=|sr'|=|r'| $ or  $ \det M_r = \det M_{r'}  \mod p^k$  and consequently we may identify the cusps of  $ X_{ns}(p^k) $ with $ (\mathbb{Z}/p^k\mathbb{Z})^* $. For the same reason $ \pm M_r, \pm M_{r'} \in C_{ns}(p^k)/\{\pm I\} $ are indistinguishable in $ X^+_{ns}(p^k) $ if and only if they were already indistinguishable in  $ X_{ns}(p^k) $ or there exists $ s' \in  (\mathbf{O}_K / p^k \mathbf{O}_K)^*  $ with $ |s'|=-1 $ such that $\pm r = \pm s'\overbar{r'} $  that is equivalent to say $ |r|=|s'\overbar{r'}|=-|r'| $ or $ \det M_r = - \det M_{r'}  \mod p^k$. In conclusion we may identify the cusps of  $ X^+_{ns}(p^k) $ with $ H=(\mathbb{Z}/p^k\mathbb{Z})^*/\{\pm 1\} $.  
\end{proof}  
 
\noindent
Furthermore, we can deduce that the covering $ \pi: X_{ns}(p^k) \to  X(p^k) $ is not ramified above the cusps. So the ramification degree of a cusp of $ X_{ns}(p^k) $ under the covering projection $ \pi':  X_{ns}(p^k) \to SL_2(\mathbb{Z})\textbackslash \mathbb{H}^* $, is equal to the one of a cusp of $ X(p^k) $ respect to $ \pi'':  X(p^k) \to SL_2(\mathbb{Z})\textbackslash \mathbb{H}^* $ that is $ p^k $. The same happens for $X^+_{ns}(p^k)  $.  

\section{Modular units on non-split Cartan curves}
Let $ t \in ((\mathbf{O}_K / p^k \mathbf{O}_K)^*/\{\pm 1\}) $: write it in the form $ t=a_1 + \sqrt{\epsilon}a_2 $ choosing $ a_1, a_2 \in \mathbb{Z} $ such that $ 0 \le a_1 \le \frac{p^k-1}{2} $, $ 0 \le a_2 \le p^k-1$ and $a_2 \le \frac{p^k-1}{2}$   if   $ a_1=0$. Define: 
$$ [t] := \frac{1}{p^k}(a_1,a_2). $$
If $ s \in (\mathbf{O}_K / p^k \mathbf{O}_K)^* $ we define $ [s]:=[\{\pm s\}] $.
Notice that if $ s,t \in (\mathbf{O}_K / p^k \mathbf{O}_K)^*  $, $ |s|=1 $ and $\gamma_s \in \Gamma_{ns}(p^k) $ lifts $M_s $ we have: 
 $$ [t] \gamma_s - [ts] \in \mathbb{Z}^2 \mbox{ or }   [t] \gamma_s + [ts] \in \mathbb{Z}^2 . $$ 
\noindent  
Analogously if $ |s|=-1 $ and $ \gamma  $ lifts $ M_sC $ to $ \Gamma^+_{ns}(p^k) $ we have:
 $$ [t] \gamma - [\overbar{ts}] \in \mathbb{Z}^2 \mbox{ or }   [t] \gamma + [\overbar{ts}] \in \mathbb{Z}^2. $$ 
These relations together with Proposition \ref{Kleinforms} imply:

\begin{prop}\label{indici} 
 The Klein forms: $ \mathfrak{k}_{[t]\gamma_s}(\tau)  $ and $ \mathfrak{k}_{[ts]}(\tau)  $ up to a $ 2p^k-$th root of unity represent the same function in the sense that:
 $$  \mathfrak{k}_{[t]\gamma_s}(\tau) = c \mathfrak{k}_{[ts]}(\tau)  $$
for some $ c \in \bm{\mu_{2p^k}} $.
 Similarly, for the Klein forms $ \mathfrak{k}_{[t]\gamma}(\tau)  $ and  $ \mathfrak{k}_{[\overbar{ts}]}(\tau)  $ we have:
  $$ \mathfrak{k}_{[t]\gamma}(\tau) = c' \mathfrak{k}_{[\overbar{ts}]}(\tau)  $$
 for some $ c' \in \bm{\mu_{2p^k}} $. 
\end{prop} 

\noindent
 For $ h \in (\mathbb{Z}/p^k\mathbb{Z})^* $ we define the following complex-valued functions on $ \mathbb{H} $:
$$ T_h (\tau) := \prod_{t \in ((\mathbf{O}_K / p^k \mathbf{O}_K)^*/\{\pm 1\}) , |t|=h}\mathfrak{k}_{[t]}(\tau),  $$ 
$$ G_h(\tau) := T_h(\tau)(\Delta(\tau))^{p^{k-1}\frac{p+1}{24}} = \prod_{t \in ((\mathbf{O}_K / p^k \mathbf{O}_K)^*/\{\pm 1\}) , |t|=h}g_{[t]}(\tau).  $$
   
\noindent
For $ h \in (\mathbb{Z}/p^k\mathbb{Z})^*/\{ \pm 1 \} $ consider:
$$ T^+_h (\tau) := \prod_{t \in ((\mathbf{O}_K / p^k \mathbf{O}_K)^*/\{\pm 1\}) , \pm|t|= h}\mathfrak{k}_{[t]}(\tau), $$  
$$ G^+_h(\tau) := T^+_h(\tau)(\Delta(\tau))^{p^{k-1}\frac{p+1}{12}} = \prod_{t \in ((\mathbf{O}_K / p^k \mathbf{O}_K)^*/\{\pm 1\}) , \pm|t|= h}g_{[t]}(\tau) .$$

\begin{prop}\label{FunzioniG} Let $ p \not= 2,3 $ a prime. Consider:

 $$ g(\tau) = \prod_{x \in  ((\mathbf{O}_K / p^k \mathbf{O}_K)^*/\{\pm 1\}) )  } g_{[x]}^{m(x)}(\tau)   $$ 
\noindent   
and suppose that it is a modular unit on $ X(p^k) $ (or equivalently that it satisfies the conditions of Theorem \ref{unitàmodulari}). If $ g(\tau) $ is a modular unit on $ X_{ns}(p^k) $ there exist integers $ \{n_h\}_{h \in (\mathbb{Z}/p^k\mathbb{Z})^* }$ such that: 
$$ g(\tau)= \prod_{h \in (\mathbb{Z}/p^k\mathbb{Z})^* }G_h^{n_h}(\tau) .$$

\noindent
Similarly, if the function $ g(\tau) $ is a modular unit on $ X^+_{ns}(p^k) $, there exist integers $ \{n^{+}_h\}_{h \in (\mathbb{Z}/p^k\mathbb{Z})^*/\{\pm 1 \}}$ such that:
$$ g(\tau)= \prod_{h \in  ((\mathbb{Z} / p^k \mathbb{Z})^*/\{\pm 1\}) }{{G^+_h}^{n^+_h}(\tau)} .$$ 
   \end{prop}

\begin{proof} 
We look for conditions on the exponents $ \{m(x)\}_{x \in   ((\mathbf{O}_K / p^k \mathbf{O}_K)^*/\{\pm 1\})} $ that guarantee:    
$$ \frac{g(\sigma^{-1}(\tau))}{g(\tau)} \in \mathbb{C} \mbox{ for every } \sigma \in  \Gamma_{ns}(p^k) \mbox{ (respectively } \Gamma^+_{ns}(p^k) \mbox{)} . $$ 
\noindent
From Proposition \ref{Kleinforms}, assertion (3), the fact that $ \Delta(\tau) $ is weakly modular of weight 12 and that by hypotesis $ 12 $ divides $ \sum{m(x)} $ we have:
$$ g(\sigma^{-1}(\tau))  = (\Delta(\sigma^{-1} (\tau)))^{\frac{1}{12}\sum{m(x)}} \prod{\mathfrak{k}_{[x]}^{m(x)}(\sigma^{-1}(\tau))} =   $$
$$ = (\Delta(\tau))^{\frac{1}{12}\sum{m(x)}} \prod{\mathfrak{k}_{[x]\sigma^{-1}}^{m(x)}(\tau)}  . $$

\noindent
By Proposition \ref{strutturanonsplit}, $ C'_{ns}(p^k)$ is a cyclic group with $ (p+1)p^{k-1} $ elements. Let $ M_r $ be a generator where $ r $ is a generator of:
$$ \{ s \in (\mathbf{O}_K / p^k \mathbf{O}_K)^*  \mbox{ with } |s|= 1 \} .$$ 
\noindent 
Every $S \in C'^+_{ns}(p^k) \setminus C'_{ns}(p^k)  $ is of the form $ M_tC$ where $ t \in (\mathbf{O}_K / p^k \mathbf{O}_K)^*  $ and $ |t|=-1 $. Fix $ S $ and choose $ \gamma_r $ lifting $ M_r $ in $ \Gamma_{ns}(p^k) $ and $\gamma_t $ lifting $ M_t $ in $ GL_2(\mathbb{Z}) $ with $ \det \gamma_t=-1 $. Of course $ \gamma_tC $ lifts $ S $ in $ \Gamma^+_{ns}(p^k) $.
\\
For every $ j $ we have that:
$$ (([x] \mbox{ mod } \mathbb{Z}^2)/\{\pm 1\}) \longmapsto (([xr^j] \mbox{ mod } \mathbb{Z}^2)/\{\pm 1\})  \mbox{ and} $$ 
$$ (([x] \mbox{ mod } \mathbb{Z}^2)/\{\pm 1\}) \longmapsto (([\overbar{xr^jt}] \mbox{ mod } \mathbb{Z}^2)/\{\pm 1\}) $$

\noindent
are permutations  of the primitive elements in $ ((\frac{1}{p^k}\mathbb{Z})^2 \mod \mathbb{Z}^2 )/(\pm 1) $.

\noindent
 As a consequence of these observations and Proposition \ref{indici}, taking $ \sigma=(\gamma_r)^j $ we have:
$$ (\Delta(\tau))^{\frac{1}{12}\sum{m(x)}} \prod{\mathfrak{k}_{[x]\gamma_r^{-j}}^{m(x)}(\tau)} = (\Delta(\tau))^{\frac{1}{12}\sum{m(x)}} \prod{\mathfrak{k}_{[xr^j]\gamma_r^{-j}}^{m(xr^j)}(\tau)} =    $$    
$$ = c_j(\Delta(\tau))^{\frac{1}{12}\sum{m(x)}} \prod{\mathfrak{k}_{[x]}^{m(xr^j)}(\tau)} = c_j \prod{g_{[x]}^{m(xr^j)}(\tau)} ,$$
  
\noindent 
where $ \{c_j\}_j $ are $ 2p^k-$th roots of unity. Taking $ \sigma = (\gamma_r)^j \gamma_t C $ we obtain:
$$ (\Delta(\tau))^{\frac{1}{12}\sum{m(x)}} \prod{\mathfrak{k}_{[x]C\gamma_t^{-1} \gamma_r^{-j}}^{m(x)}(\tau)} = (\Delta(\tau))^{\frac{1}{12}\sum{m(x)}} \prod{\mathfrak{k}_{[\overbar{xr^jt}]C\gamma_t^{-1} \gamma_r^{-j}}^{m(\overbar{xr^jt})}(\tau)} = $$
$$  = d_j (\Delta(\tau))^{\frac{1}{12}\sum{m(x)}} \prod{\mathfrak{k}_{[x]}^{m(\overbar{xr^jt})}(\tau)} = d_j \prod{g_{[x]}^{m(\overbar{xr^jt})}(\tau)} ,$$

\noindent 
where $ \{d_j\}_j $ are  $ 2p^k-$th roots of unity. Consider the following expression:
$$  \frac{g(\gamma_r^{-1}(\tau))}{g(\tau)} = c_1 \prod{g^{m(xr)-m(x)}_{[x]}(\tau) } .$$
\noindent
By the independence of Siegel functions \cite[p.42 or p.120]{KL} a product $ \prod {g_{[x]}^{l(x)}}  $ is constant if and only if the exponents $ l(x)$ are all equal. So the previous quotient is constant if and only if:
 $$ a(xr^j)= m(xr^{j+1})-m(xr^j) \mbox{ satisfy } a(xr^j)=a(xr^l) \mbox{  } \mbox{ for all } j,l \in \mathbb{Z}.   $$ 
   
\noindent  
But $ (\gamma_r)^{\frac{p+1}{2} p^{k-1}} \equiv -I \mod p^k$ and $ r^{\frac{p+1}{2} p^{k-1}} = -1 \mod p^k $. So we have that 
$ \sum_{j=1}^{ \frac{p+1}{2} p^{k-1}}a(xr^j) =0 $ and consequently $ a(xr^j)=0 $ for every $ j $, which implies that $ m(xr^j)$ does not depend on $ j $. 
Since $g(\tau)$ is $ \Gamma(p^k)-$invariant and every element in $ \Gamma_{ns}(p^k) $ can be written in the form $ \gamma\gamma_r^j $ with $ \gamma \in  \Gamma(p^k) $, we conclude that if $ g(\sigma^{-1}(\tau))/g(\tau) \in \mathbb{C} \mbox{ for every } \sigma \in  \Gamma_{ns}(p^k)  $, this implies that if $ |x|=|y|$ then $ m(x)=m(y)$. For each $ h $ invertible mod $p^k $ choose $ x $ with $ |x|=h $, put $ n_h:=m(x)$ and the first claim follows. 

 Consider now:
 $$  \frac{g((C\gamma_t^{-1})(\tau))}{g(\tau)} = d_0 \prod{g^{m(\overbar{xt})-m(x)}_{[x]}(\tau) } .$$

\noindent
If this quotient is constant the exponent of $ g_{[x]}(\tau) $ is equal to the exponent of the Siegel function $ g_{[\overbar{x}rt]}(\tau) $. So:
$$ m(\overbar{xt})-m(x) = m(x\overbar{rt^2})-m(\overbar{x}rt)  $$
or equivalently:
$ m(\overline{xt}) + m(\overbar{x}rt) = m(x) + m(x\overbar{rt^2}) $.
But $ |rt\overbar{t^{-1}}|=1  $ so $ m(\overline{xt}) = m(\overbar{x}rt) $  and $ |\overbar{rt^2}|=1 $ so $  m(x) = m(x\overbar{rt^2}) $. Hence $ m(x)=m(\overbar{xt}) $ and observe that $ |x|=-|\overbar{xt}| $. So, in consideration of the previous result, we can conclude that $ g(\sigma^{-1}(\tau))/g(\tau) \in \mathbb{C} \mbox{ for every } \sigma \in  \Gamma^+_{ns}(p^k)  $  implies that if $ |x|=|y|$ or $ |x|=-|y| $ then $m(x)=m(y)$.
For every $ h \in (\mathbb{Z} / p^k \mathbb{Z})^*/\{\pm 1\} $ choose $ x  $ such that $ \pm|x|=h $ and define $ n^+_h := m(x) $ and the second claim follows. 
\end{proof}   

\begin{prop}\label{Somme} The product:
$$ \prod_{h \in (\mathbb{Z}/p^k\mathbb{Z})^* }T_h^{n_h}(\tau) $$
is a nearly holomorphic modular form for $ \Gamma(p^k) $ if and only if  $p $ divides $ \sum_h n_h h$. 
\end{prop}
\begin{proof} First of all, for every $ h $ invertible mod $ p^k $:
$$ \mbox{ (1) } \sum_ { \begin{scriptsize} \begin{array}{c} \pm s \in ((\mathbf{O}_K / p^k \mathbf{O}_K)^*/\{\pm 1\})  \\  |\pm s|=h \end{array} \end{scriptsize} }\left(\frac{1}{2}(s+\overline{s})\right)^2 =  \frac{h}{4}(p+1)p^{k-1} \mod p^k,   $$
$$ \mbox{ (2) } \sum_{ \begin{scriptsize} \begin{array}{c} \pm s \in ((\mathbf{O}_K / p^k \mathbf{O}_K)^*/\{\pm 1\}
)  \\  |\pm s|=h \end{array} \end{scriptsize} }\left(\frac{1}{2\sqrt{\epsilon}}(s-\overline{s})\right)^2 = -  \frac{h}{4\epsilon}(p+1)p^{k-1} \mod p^k,   $$
$$ \mbox{ (3) } \sum_{ \begin{scriptsize} \begin{array}{c} \pm s \in ((\mathbf{O}_K / p^k \mathbf{O}_K)^*/\{\pm 1\})  \\  |\pm s|=h \end{array} \end{scriptsize} }\left(\frac{1}{2}(s+\overline{s})\right)\left(\frac{1}{2\sqrt{\epsilon}}(s-\overline{s})\right) = 0 \mod p^k .  $$
We prove only the first assertion because the other statements can be shown by the same argument . Every $ s \in (\mathbf{O}_K / p^k \mathbf{O}_K)^* $ with $ |s|= h $ can be written as $ s=r^i\alpha_h $ where $ r $ is a generator of the subgroup $ \{ t \in (\mathbf{O}_K / p^k \mathbf{O}_K)^*  \mbox{ with } |t|= 1 \} $ and $ \alpha_h $ are fixed elements such that $ |\alpha_h|=h $.
$$ \sum_{|\pm s|=h}(\frac{1}{2}(s+\overline{s}))^2 = \sum_{i=0}^{\frac{p+1}{2}p^{k-1}-1}(\frac{1}{2}(r^i\alpha_h +\overline{ r^i\alpha_h}))^2 = $$  $$ =
\displaystyle\frac{\alpha_h^2}{4}\sum_ {i=0}^{\frac{p+1}{2}p^{k-1}-1}(r^{2i}) +
 \displaystyle\frac{\overline{\alpha_h^2}}{4}\sum_ {i=0}^{\frac{p+1}{2}p^{k-1}-1}(r^{-2i}) + \frac {\alpha_h\overbar{\alpha_h}}{4} (p+1)p^{k-1}
$$ 
\noindent
and the assertion (1) follows because:
$$\displaystyle \sum_{i=0}^{\frac{p+1}{2}p^{k-1}-1}r^{2i}= 
\displaystyle \sum_{i=0}^{\frac{p+1}{2}p^{k-1}-1}r^{-2i} =
\displaystyle\frac{1- r^{(p+1)p^{k-1}} }{1-r^2} =0 \mod p^k  $$
\noindent 
To prove this proposition we apply Proposition \ref{Kleinforms} to the product $ \prod_{h}T_h^{n_h}(\tau) $. Considering that for every $ s= a_1+\sqrt{\epsilon}a_2 \in (\mathbf{O}_K / p^k \mathbf{O}_K)^* $ we have:
$$ p^k[ s] \equiv (a_1,a_2) \mbox{ mod } (p^k\mathbb{Z})^2 \mbox{ or } p^k[ s] \equiv -(a_1,a_2) \mbox{ mod } (p^k\mathbb{Z})^2 \mbox{   and:} $$ $$ (a_1,a_2) \equiv \left( \frac{1}{2}(s+\overline{s}),\frac{1}{2\sqrt{\epsilon}}(s-\overline{s}) \right)  \mbox{ mod }  (p^k\mathbb{Z})^2  $$
and reformulating condition (5) of Proposition \ref{Kleinforms} in terms of assertions (1),(2) and (3) we attain the desired result. 
\end{proof}  
 
From this proposition it follows immediately that the functions $ T^+_h({\tau}) $ are nearly holomorphic for $ \Gamma(p^k) $. We will examine them further in details.

For every $ s= \begin{pmatrix} a&b\\c&d \end{pmatrix} \in SL_2({\mathbb{Z}}) $ define:
$$ J_s(\tau)=(c\tau+d)^{-(p+1)p^{k-1}} , \tau \in \mathbb{H}. $$

\noindent 
\begin{prop}\label{piùmenodiedrale} 
For every prime $ p \equiv 3 $ mod $ 4 $,  for every $ h \in ((\mathbb{Z}/p^k\mathbb{Z})^*/\{\pm 1\}) $ and for every $ s \in \Gamma^+_{ns}(p^k) $ we have: 
$$ T^+_h(s(\tau))=J_s(\tau) T^+_h(\tau)  $$
in other words $ T^+_h(\tau) $ is a nearly holomorphic modular form for $ \Gamma^+_{ns}(p^k) $ of weight $-(p+1)p^{k-1}  $.\\
\noindent
If $ p \equiv 1 $ mod $ 4 $ and  $ s \in \Gamma_{ns}(p^k) $ we have:
$$ T^+_h(s(\tau))=J_s(\tau) T^+_h(\tau)  $$
in other words $ T^+_h(\tau) $ is a nearly holomorphic modular form for $ \Gamma_{ns}(p^k) $ of weight $ -(p+1)p^{k-1} $.\\
\noindent
If $ p \equiv 1 $ mod $ 4 $ and  $ s \in \Gamma^+_{ns}(p^k) \setminus \Gamma_{ns}(p^k) $ we have:
$$ T^+_h(s(\tau))= - J_s(\tau) T^+_h(\tau) .$$
\end{prop}

\begin{proof} 

It is clear from Proposition \ref{Kleinforms} that for every $ s\in \Gamma^+_{ns}(p^k) $ there exists a $ 2p^k-$th root of unity $ c $ such that:
$ T^+_h(s(\tau))=c  T^+_h(\tau)J_s(\tau) $ so it is natural to define: 
$$ C_h(s) = \displaystyle\frac{T^+_h(s(\tau))}{T^+_h(\tau)J_s(\tau)} \in \bm{\mu_{2p^k}}. $$
\noindent
On the one hand:
$$ T^+_h((ss')(\tau)) = C_h(ss')T^+_h(\tau)J_{ss'}(\tau),  $$
on the other hand:
$$  T^+_h(s(s'(\tau)))= C_h(s)T^+_h(s'(\tau))J_s(s'(\tau)) = C_h(s)C_h(s')J_s(s'(\tau))J_{s'}(\tau)T^+_h(\tau).  $$

\noindent
Considering that $ J_{ss'}(\tau) = J_s(s'(\tau))J_{s'}(\tau) $ we have:
$$ C_h(ss')=C_h(s)C_h(s'). $$ 
From Proposition \ref{Somme} we deduce easily that $ C_h(\pm\Gamma(p^k))=1 $ for every $ h $. So $C_h $ are characters of $ \Gamma^+_{ns}(p^k)/\pm \Gamma(p^k) $. Since this quotient is isomorphic to $ C'^+_{ns}(p^k)/\{\pm I\} $ and since for every $ \alpha \in  C'^+_{ns}(p^k) \setminus C'_{ns}(p^k) $ we have $ \alpha^2 = -I $, by Proposition \ref{strutturanonsplit} we obtain that $ \Gamma^+_{ns}(p^k)/\pm \Gamma(p^k)   $  
  is a dihedral group of $ (p+1)p^{k-1} $ elements. These observations entail \textit{ipso facto} that $ C_h(s) \in \{ \pm 1 \} $. As in Proposition \ref{FunzioniG} choose a matrix $ \gamma_r \in SL_2(\mathbb{Z}) $ lifting $ M_r \in C'_{ns}(p^k) $ where $ r $ generates the subgroup of $ (\mathbf{O}_K / p^k \mathbf{O}_K)^* $ of elements of norm 1.\\ Choose $ \gamma = \begin{pmatrix} a&b\\c&d \end{pmatrix} $ in $ \Gamma^+_{ns}(p^k) \setminus \Gamma_{ns}(p^k)$  . It is not restrictive to suppose that $ a=d \mbox{ mod } 2 $. If this did not happen we would alternatively choose: $$   \gamma = \begin{pmatrix} a&b\\c&d \end{pmatrix} \begin{pmatrix} 1&p^k\\0&1 \end{pmatrix} = \begin{pmatrix} a&{ap^k+b}\\c&{cp^k+d} \end{pmatrix}.$$
   If $ a\not\equiv d \mbox{ mod } 2$ then $ b $ and $ c $ are inevitably odd because $ ad-bc=1 $ so the new coefficients on the diagonal verify $ a \equiv cp^k+d \mbox{ mod } 2$. \\
Notice that  $  \{ \gamma\gamma_r^j\gamma^{-1} \mbox{, }  \gamma\gamma_r^j  \}_{j=1,...,\frac{p+1}   {2}(p^{k-1})}$ is a set of representatives of cosets for the quotient group  $ \Gamma^+_{ns}(p^k)/\pm \Gamma(p^k) $. Furthermore if  $ h \in (\mathbb{Z}/p^k\mathbb{Z})^*/\{\pm 1\} $ and  $s \in  (\mathbf{O}_K / p^k \mathbf{O}_K)^*/\{\pm 1\} $ with $ \pm|s|=h $, there exists a $ 2p^k-$th root of unity $ c' $ such that:

$$ T^+_h(\tau) = c' \prod_{j=1}^{\frac{p+1}{2}p^{k-1}} {\mathfrak{k}_{[s]\gamma\gamma_r^j\gamma^{-1}}}(\tau)\prod_{j=1}^{\frac{p+1}{2}p^{k-1}} {\mathfrak{k}_{[s]\gamma\gamma_r^j}}(\tau).  $$
\noindent
We calculate:
 $$ T^+_h(\gamma(\tau)) = c' J_{\gamma}(\tau)  \prod_{j=1}^{\frac{p+1}{2}p^{k-1}} {\mathfrak{k}_{[s]\gamma\gamma_r^j}}(\tau)\prod_{j=1}^{\frac{p+1}{2}p^{k-1}} {\mathfrak{k}_{[s]\gamma\gamma_r^j\gamma}}(\tau) =  $$
 $$ = c' (-1)^{\frac{p+1}{2}} J_{\gamma}(\tau)  \prod_{j=1}^{\frac{p+1}{2}p^{k-1}} {\mathfrak{k}_{[s]\gamma\gamma_r^j}}(\tau)\prod_{j=1}^{\frac{p+1}{2}p^{k-1}} {\mathfrak{k}_{-[s]\gamma\gamma_r^j\gamma}}(\tau)  .$$
\noindent 
But $ \gamma\gamma_r^j\gamma^{-1} \equiv -\gamma\gamma_r^j\gamma  $ mod $ p^k $ and $ \gamma^{-1}+\gamma $ (in agreement with the previous convention) has all even coefficients so:
$$ [s]\gamma\gamma_r^j\gamma^{-1} - (-[s]\gamma\gamma_r^j\gamma) = [s]\gamma\gamma_r^j( \gamma^{-1}+\gamma) \in (2\mathbb{Z})^2   $$
and considering Proposition \ref{Kleinforms} part (2) we have:
$$ \frac{{\mathfrak{k}_{-[s]\gamma\gamma_r^j\gamma}}(\tau) }{{\mathfrak{k}_{[s]\gamma\gamma_r^j\gamma^{-1}}}(\tau)}  \in \bm{\mu_{p^k}}. $$
Therefore: $$ C_h(\gamma) = (-1)^{\frac{p+1}{2}}\prod_{j=1}^{\frac{p+1}{2}p^{k-1}}\frac{{\mathfrak{k}_{-[s]\gamma\gamma_r^j\gamma}}(\tau) }{{\mathfrak{k}_{[s]\gamma\gamma_r^j\gamma^{-1}}}(\tau)},  $$ 

\noindent
so $ C_h(\gamma)(-1)^{\frac{p+1}{2}} \in \bm{\mu_{p^k}} \cap \{\pm 1\} $, we have necessarily $ C_h(\gamma)=(-1)^{\frac{p+1}{2}} $ for every $ \gamma \in \Gamma^+_{ns}(p^k) \setminus \Gamma_{ns}(p^k) $ and the proposition follows. 
  \end{proof}
 
\begin{thm}\label{powprod} If $ p \not= 2,3 $ the subgroup of modular units in $ F_{p^k} $ of $ X^+_{ns}(p^k) $ consists (modulo constants) of power products: 
$$  g(\tau)= \prod_{h \in  ((\mathbb{Z} / p^k \mathbb{Z})^*/\{\pm 1\}) }{{G^+_h}^{n^+_h}(\tau)} $$
where $ d=\displaystyle\frac{12}{\gcd(12,p+1)} \mbox{ divides }  \sum_{h}n^+_h $. 

\end{thm}  
 
\begin{proof}
By Proposition \ref{FunzioniG} and Theorem \ref{unitàmodulari}, every modular unit on $ X^+_{ns}(p^k) $ can be written in the above indicated way.
In fact, $ d|\sum_h{n^+_h} $ is equivalent to saying: $ 12|(p+1)p^{k-1}\sum_h{n^+_h} $.
\\By Proposition \ref{piùmenodiedrale} all the functions of this form are modular units of $ X^+_{ns}(p^k) $. In fact, if $ p \equiv 3 \mbox{ mod } 4 $, the functions $ T^+_h(\tau) $ are nearly holomorphic modular forms for $ X^+_{ns}(p^k) $. If $ p \equiv 1 \mbox{ mod } 4 $, even if the functions $ T^+_h(\tau) $ are not nearly holomorphic modular forms for $ X^+_{ns}(p^k) $, the product $ \prod_{h }{{T^+_h}^{n^+_h}(\tau)} $ has this property, because $ \sum_h{n^+_h} $ is even in this case.  
\end{proof} 

Notice that such a writing for $ g(\tau) $ is not unique because of the fact that the following product is constant: $$ \displaystyle
\prod_{h \in ((\mathbb{Z} / p^k \mathbb{Z})^*/\{\pm 1\})}{G^+_h(\tau)} = \prod_{t \in   (\mathbf{O}_K / p^k \mathbf{O}_K)^*/\{\pm 1\}}{g_{[t]}(\tau)}.$$

\begin{rem}\label{pseudoGalois} Let $g$ be a generator of $ (\mathbb{Z}/p^k\mathbb{Z})^* $. Choose $s  \in (\mathbf{O}_K / p^k \mathbf{O}_K)^* $ with $ |s|=g $ and denote with $ \rho \in Gal(F_{p^k},\mathbb{Q}(j))$ the automorphism corresponding to the matrix $ M_s $ respect to the isomorphism  $ Gal(F_{p^k},\mathbb{Q}(j)) \cong GL_2(\mathbb{Z}/p^k\mathbb{Z})/{\pm I} $  described in Theorem \ref{fricke}. Let $ F^+_{ns}(p^k) $ be the subfield of $ F_{p^k} $ fixed by $ C'^+_{ns}(p^k)/\pm I $.
Choose $\sigma \in  Gal(F_{p^k},\mathbb{Q}(j))  $.
From Galois theory we have:
$$  Gal(F_{p^k}, \sigma ( F^+_{ns}(p^k)) ) = \sigma Gal(F_{p^k}, F^+_{ns}(p^k) )\sigma^{-1}  $$
\noindent
thus saying that $ \sigma ( F^+_{ns}(p^k)) = F^+_{ns}(p^k)  $ amounts to say that $ \sigma $ belongs to the normalizer of $ C'^+_{ns}(p^k)/\pm I $, in other words we have: $\sigma \in C^+_{ns}(p^k) / \pm I $. Consider $ \sigma_1, \sigma_2 \in C^+_{ns}(p^k) / \pm I $. We have $ \sigma_1(f(\tau)) = \sigma_2 (f(\tau))  $ for every $ f(\tau) \in F^+_{ns}(p^k) $ if and only if $  \sigma_1\sigma_2^{-1} \in C'^+_{ns}(p^k)/\pm I $ or equivalently $ \det \sigma_1 = \det \sigma_2  $. So every automorphism $ \sigma\restriction_{F^+_{ns}(p^k)}: F^+_{ns}(p^k)\rightarrow F^+_{ns}(p^k)$ fixing $ \mathbb{Q}(j) $ can be written in the form $ \sigma = \rho^j $ for some $ 0 \le j \le \varphi(p^k)-1  $.  
 Notice that if $$ f(\tau)= \prod_{h \in  ((\mathbb{Z} / p^k \mathbb{Z})^*/\{\pm 1\})  }{{G^+_h}^{n^+_h}(\tau)} $$
 and $$ h(\tau)= \prod_{h \in  ((\mathbb{Z} / p^k \mathbb{Z})^*/\{\pm 1\})  }{{G^+_{h(\pm g)}}^{n^+_h}(\tau)} $$
\noindent are modular units on $ X^+_{ns}(p^k) $, from proposition \ref{sigfun} we have: 
$$ (\rho(f(\tau)))^{12p^k} = \rho(f(\tau)^{12p^k})= \rho \left(\prod_{h \in  ((\mathbb{Z} / p^k \mathbb{Z})^*/\{\pm 1\})  }{{G^+_h}^{12p^kn^+_h}(\tau)}\right) =$$  
$$=  \prod_{h \in  ((\mathbb{Z} / p^k \mathbb{Z})^*/\{\pm 1\})  }{{G^+_{h(\pm g)}}^{12p^kn^+_h}(\tau)} = (h(\tau))^{12p^k}. $$ 
So $ \rho(f(\tau)) = c h(\tau)$  for some $ c \in \mathbb{Q}(\zeta_{p^k})$ and all the functions $ \rho^j(f(\tau)) $ are modular units. Choosing $ j=  \frac{1}{2}  \varphi(p^k) $ we deduce that for every modular unit $ f(\tau) $ on  $ X^+_{ns}(p^k) $ there exist $  c' \in \mathbb{Q}(\zeta_{p^k}) $ such that:
$$ c'f(\tau) \in \mathbb{Q}\left(\cos \left(\frac{2\pi}{p^k}\right)\right)((q^{p^{-k}})) \mbox{ with } q=e^{2\pi i \tau}. $$ 

\end{rem} 
 
\section{Cuspidal Divisor Class Group of non-split Cartan curves} 
Let $ p \ge 5 $ a prime and let $  R = \mathbb{Z}[H]  $ be the group ring of $ H=(\mathbb{Z}/p^k\mathbb{Z} )^*/\{ \pm 1 \} $ over $  \mathbb{Z} $.
Let $ w $ be a generator of $ H $. For $ \alpha \in \mathbb{Z}/p^k\mathbb{Z} $, let be $ a \in \mathbb{Z} $ congruent to $ \alpha \mbox{ mod }p^k $. We define:
$$ \left\langle \frac{\alpha}{p} \right\rangle:= \left\langle \frac{a}{p} \right\rangle .$$
Define the Stickelberger element: 
$$ \theta= \displaystyle\frac{p^k}{2} \sum_{i=1}^{\frac{p-1}{2}p^{k-1}} {\displaystyle\sum_{ \begin{scriptsize} \begin{array}{c} s \in ((\mathbf{O}_K / p^k \mathbf{O}_K)^*/\{\pm 1\})  \\  \pm| s|=w^i \end{array} \end{scriptsize} } B_2 \left( \left\langle \frac{\frac{1}{2}(s+\overline{s}) }{p^k} \right\rangle \right) } w^{-i} \in \mathbb{Q}[H].  $$
Define the ideals:
$$ R_0 := \Big{\{} \sum b_jw^j \in R \mbox{ such that } \deg \left( \sum b_jw^j\right)=\sum b_j=0 \Big{\}}, $$
$$  R_d := \Big{\{} \sum b_jw^j \in R \mbox{ such that } d \mbox{ divides } \deg\left(\sum b_jw^j\right)=\sum b_j \Big{\}}. $$ 
Now we can state the main result:
\begin{mainthm}\label{main}
The group generated by the divisors of modular units in $F_{p^k}  $ of the curve $ X^+_{ns}(p^k) $ can be expressed both as $ R_d\theta $ and as Stickelberger module $ R\theta \cap R $. The Cuspidal Divisor Class Group on $ X^+_{ns}(p^k) $ is a module over $Z[H]$ and we have the following isomorphism: 
$$ \mathfrak{C}^+_{ns}(p^k) \cong R_0 / R_d \theta. $$ 
\end{mainthm}
\begin{proof} 
 For every $ i \in \mathbb{Z}/\frac{\varphi(p^k)}{2}\mathbb{Z} $ define:
$$ a_i = \displaystyle\frac{p^k}{2}\sum_{ { \begin{scriptsize} \begin{array}{c} s \in ((\mathbf{O}_K / p^k \mathbf{O}_K)^*/\{\pm 1\})  \\  \pm| s|=w^i \end{array} \end{scriptsize} }} B_2 \left( \left\langle \frac{\frac{1}{2}(s+\overline{s}) }{p^k} \right\rangle \right). $$
\noindent
We identify the cusps of $ X^+_{ns}(p^k) $ with the elements in $H= (\mathbb{Z}/p^k\mathbb{Z})^*/\{\pm 1\} $ as explained in Proposition \ref{numerocuspidi}. In consideration of Proposition \ref{divisori} we obtain:
$$ \mbox{ div } {G^+_{w^j}}^d(\tau) = d \sum_{i=1}^{\frac{p-1}{2}p^{k-1}} a_i w^{j-i}. $$
If $ p\not\equiv 11 \mbox{ mod } 12$, the function $  G^+_{w^j}(\tau) $ is not $ \Gamma^+_{ns}(p^k)-$invariant but with a slight abuse of notation we write:
$$ \mbox{ div } G^+_{w^j}(\tau) = \sum_{i=1}^{\frac{p-1}{2}p^{k-1}} a_i w^{j-i}. $$
 It is clear that $\mbox{div }G^+_{w^i}(\tau) \in \mathbb{Q}[H] $ and  $d \mbox{ div }G^+_{w^i}(\tau) \in R $.
Consider the Stickelberger element:
$$  \mbox{div } G^+_{\{\pm 1\}}(\tau) =  \sum_{i=1}^{\frac{p-1}{2}p^{k-1}} a_i w^{-i} = $$  
$$ = \displaystyle\frac{p^k}{2} \sum_{i=1}^{\frac{p-1}{2}p^{k-1}} {\displaystyle\sum_{ \pm|s|=w^i} B_2 \left( \left\langle \frac{\frac{1}{2}(s+\overline{s}) }{p^k} \right\rangle \right) } w^{-i} = \theta.  $$
Notice that: $ \mbox{div } G^+_{w^j}(\tau) = w^j\theta  $. By Theorem \ref{powprod}, a $ \Gamma^+_{ns}(p^k)-$invariant function $g(\tau) \in F_{p^k} $ is a modular unit of $ X^+_{ns}(p^k) $, if and only if $ \mbox{ div } g(\tau) \in R_d \theta $. By \cite[Proposition 2.3, Chapter 5]{KL} we have $R_d \theta =R\theta \cap R  $.
\end{proof}
\noindent

\begin{rem} Following Remark \ref{pseudoGalois}, consider $  G:=\displaystyle\frac{C^+_{ns}(p^k) / \pm I}{C'^+_{ns}(p^k) / \pm I} \cong (\mathbb{Z}/p^k\mathbb{Z})^*$  and let $ \rho $ be a generator of $ G/\{\pm 1\} $ with $ \pm \det \rho = w $. We may identify the group $ H $ parameterizing the cusps of $ X^+_{ns}(p^k) $ with $ G/\{\pm 1\} $ observing that for every automorphism $ \rho^j \in G/\{\pm 1\} $   and each moduar unit $h(\tau) \in F^+_{ns}(p^k)$ we have:
\begin{center}
ord$_{w^{-j}} (h(\tau)) = $ord$_{\rho^{-j}(\infty)} (h(\tau)) = $ ord$_\infty \rho^j (h(\tau)) $
\end{center}  
and  
$$ \mbox{div}(\rho^j(h(\tau))) = w^j\mbox{div}(h(\tau)).  $$
If $ \sum a_j \rho^j \in \mathbb{Z}[G/\{\pm 1\}] \cong \mathbb{Z}[H] $ we define $$ \left(\sum a_j w^j\right) (h(\tau)) = \prod \rho^j(h(\tau))^{a_j}  $$ and clearly we have:
$$ \mbox{div} \left(\prod \rho^j(h(\tau))^{a_j}\right) =  \left(\sum a_j w^j \right)\mbox{div}(h(\tau)) $$
so $ \mathfrak{C}^+_{ns}(p^k) $ has a natural structure of $ \mathbb{Z}[H]$-module which emphasizes the analogy with the classical theory of cyclotomic fields recalled in the introductory section.
\end{rem}

\noindent
Define:
$$ \theta' = \theta - \displaystyle\frac{(p+1)p^{2k-1}}{12}\sum_{i=1}^{\frac{p-1}{2}p^{k-1}}w^i $$
\noindent
and observe that $ \theta' \in R$ and $ \deg(\theta') = - \displaystyle\frac{p^2-1}{24} p^{3k-2}. $

\begin{prop}\label{calcolorapido} We have: $$ R_0 \cap \left(R\theta' + p^{2k-1} R \sum_{i=1}^{\frac{p-1}{2}p^{k-1}} w^i \right) = R_d \theta .$$

\end{prop}
\begin{proof}
Let $ \alpha, \beta \in R $ such that:
$$ \alpha \theta' + p^{2k-1}\beta\sum_{i=1}^{\frac{p-1}{2}p^{k-1}} w^i \in R_0 .$$
\noindent Then 
$$ -\deg(\alpha)\frac{p^2-1}{24} p^{3k-2} + p^{2k-1} \deg(\beta) \displaystyle\frac{p-1}{2}p^{k-1} = 0  $$
\noindent implies $ (p+1)\deg(\alpha)= 12 \deg(\beta)  $. This is equivalent to say:
$$ d=\displaystyle\frac{12}{\gcd(12,p+1)} \mbox{ divides }\deg(\alpha) $$ and
$$ \alpha \theta' + p^{2k-1}\beta\sum w^i = \alpha\theta' + \frac{(p+1)p^{2k-1}\deg(\alpha)}{12} \sum w^i = \alpha \theta .$$
\end{proof}

\begin{thm}\label{Cardinalità}
We have:
$$ |\mathfrak{C}^+_{ns}(p^k)| = \displaystyle\frac{|\det A_{\theta'}|}{\frac{p^2-1}{24} p^{k-1} e} = 
\displaystyle 24\frac{ \displaystyle\prod{}^{}  {\frac{p^k}{2}B_{2,\chi}}}{\gcd(12,p+1)(p-1)p^{k-1}},  $$
 where $ A_{\theta'} $ is a circulant Toeplitz matrix, $ e= p^{3k-2}\frac{p-1}{2d} $ and the product runs over all nontrivial characters $ \chi $ of $ C(p^k)/{\pm I} $ such that $ \chi(M)=1 $ for every $ M \in C(p^k) $
 with $ \det M = \pm 1 $.
\end{thm}

\begin{proof}
\noindent From Proposition \ref{calcolorapido} and the following isomorphism:
$$ R_0 / \left(R_0 \cap \left(R\theta' + p^{2k-1} R \sum w^i \right)\right) \cong$$ 
 $$ \cong \left(R_0 + R\theta' + p^{2k-1} R \sum w^i \right) / \left(R\theta' + p^{2k-1} R \sum w^i\right)  $$
\noindent we deduce that
$$ \displaystyle|\mathfrak{C}^+_{ns}(p^k)|=\left(R_0 + R\theta' + p^{2k-1} R \sum w^i\right):\left(R\theta' + p^{2k-1} R \sum w^i\right).   $$
From the following chain of consecutive inclusions:
$$ R \supset R_0 + R\theta' + p^{2k-1} R \sum w^i \supset R\theta' + p^{2k-1} R \sum w^i \supset R\theta'   $$ 
\noindent we obtain
$$ \displaystyle|\mathfrak{C}^+_{ns}(p^k)| = \displaystyle\frac{(R:R\theta')}{\left( R : \left(R_0 + R\theta' + p^{2k-1} R \sum w^i\right)\right)\left(\left(R\theta' + p^{2k-1} R \sum w^i\right) : R\theta'\right)}.  $$
\noindent
Define $$ e:= \gcd \left(\deg(\theta'),p^{2k-1}\deg\left(\sum w^i\right)\right) = $$  $$ 
= p^{3k-2}\gcd \left(\frac{p^2-1}{24},\frac{p-1}{2} \right) = p^{3k-2}\frac{p-1}{2d}. $$
\noindent
It is clear that $$ R_0 + R\theta' + p^{2k-1} R \sum w^i = R_e, $$  where by $ R_e $ we mean the ideal of $ R $ consisting of elements whose degree is divisibile by $ e$. So
$$ \left(R : \left(R_0 + R\theta' + p^{2k-1} R \sum w^i\right)\right) = e. $$

\noindent Regarding $ \left(\left(R\theta' + p^{2k-1} R \sum w^i\right) : R\theta'\right) $, we observe that
$$ \left(R\theta' + p^{2k-1} R \sum w^i\right) \big{/} R\theta' \cong \left(p^{2k-1} R \sum w^i\right) \big{/} \left(p^{2k-1} R \sum w^i \cap R\theta'\right). $$
\noindent
But $ \prod_{h}{G^+_h}^{n^+_h} $ is constant if and only if all $ n^+_h $ are the same and so
$$ \mbox{ div }  \prod_{h}{G^+_h}^{n^+_h} = \left(\sum n^+_h h\right)\theta= \sum n^+_h h \left(\theta' + \frac{(p+1)p^{2k-1}}{12}\sum_{i=1}^{\frac{p-1}{2}p^{k-1}}w^i  \right) =0$$ 
implies  
$$ \left(\sum n^+_h h \right)\theta' = - \frac{(p+1)p^{2k-1}}{12} \sum n^+_h \sum w^i \iff n^+_{w}=n^+_{w^2}=n^+_{w^3}=...=n^+_{\pm 1}. $$ 
\noindent 
But
$$ \left( \sum w^i \right) \theta' = \deg (\theta') \sum w^i = - \displaystyle\frac{p^2-1}{24} p^{3k-2} \sum w^i $$ 
so: $$ \left(R\theta' + p^{2k-1} R \sum w^i\right) : R\theta' = \frac{p^2-1}{24} p^{k-1}. $$
The last index we need to compute is $ (R:R\theta') $. Write $ \theta'=\sum a'_i w^{-i} $ and $ a'_i=a_i - \frac{p+1}{12} p^{2k-1} $. Define the following matrix:
 $$ A_{\theta'} = 
  \begin{pmatrix}
   
  a'_{0} & a'_{1} & a'_{2}     & \dots & a'_{\frac{p-1}{2}p^{k-1} -2} & a'_{\frac{p-1}{2}p^{k-1} -1} \\
 a'_{\frac{p-1}{2}p^{k-1}-1} & a'_{0} & a'_{1}     &  \dots & a'_{\frac{p-1}{2}p^{k-1} -3} & a'_{\frac{p-1}{2}p^{k-1} -2}\\
 a'_{\frac{p-1}{2}p^{k-1}-2} & a'_{\frac{p-1}{2}p^{k-1}-1} & a'_{0}    &  \dots & a'_{\frac{p-1}{2}p^{k-1} -4} & a'_{\frac{p-1}{2}p^{k-1} -3}\\ 
\hdotsfor{6} \\ 
  a'_{2} & a'_{3}  & a'_{4}  & \dots & a'_{0} & a'_{1} \\
  a'_{1} & a'_{2} & a'_{3}  &  \dots & a'_{\frac{p-1}{2}p^{k-1} -1} &  a'_{0}
  \end{pmatrix}.   
   $$
\noindent
We have: $ (R:R\theta')=|\det A_{\theta'}| $. The matrix $ A_{\theta'} $ is a circulant Toeplitz matrix, in other words the coefficients $ (A_{\theta'})_{i,j} $ depend only on $ i-j \mod \frac{p-1}{2}p^{k-1}$. This is the matrix of multiplication by $ \theta'$ in $ \mathbb{C}[H] $, so we easily deduce that for $ n=1,2,...,\frac{p-1}{2}p^{k-1} $ the eighenvalues of $ A_{\theta'}$ are:
$$ \lambda_n = \sum_{j=1}^{\frac{p-1}{2}p^{k-1}}a'_j e^{\begin{Large}\frac{4 \pi i j n}{(p-1)p^{k-1}}\end{Large} }  $$
with corresponding eighenvectors:
$$ v_n = \sum_{j=1}^{\frac{p-1}{2}p^{k-1}} e^{\begin{Large}\frac{4 \pi i j n}{(p-1)p^{k-1}}\end{Large} } w^j .$$
\noindent  
Observe that $ \lambda_{\frac{p-1}{2}p^{k-1}} = \sum a'_i = \deg(\theta') = -\frac{p^2-1}{24}p^{3k-2} $ and that according to the definition of Theorem \ref{CDCG} the others $ \lambda_n $ correspond to the generalized Bernoulli number $\frac{p^k}{2} B_{2,\chi} $ where $ \chi $ runs over the nontrivial characters of $ C(p^k)/{\pm I} $ such that $ \chi(M)=1 $ for every $ M \in C(p^k) $
with $ \det M = \pm 1 $.
\\ Gathering all this information together we obtain the desired result. 
\end{proof}

\section{Explicit calculation}

In this section we examine the curve $ X^+_{ns}(p) $ more in details. Denote with $ v $ a generator of the multiplicative group of $ \mathbb{F}_{p^2} $ and indicate with $ \omega $ a generator of the character group $ \hat{\mathbb{F}_{p^2}^*} $ viewing $ C(p) \cong \mathbb{F}_{p^2}^* $. By Theorem \ref{Cardinalità}, in this case we have: 
$$ B_{2,\chi} =  \sum_{x \in \mathbb{F}_{p^2}^* / \pm 1} B_2 \left( \left\langle \frac{\frac{1}{2}\mbox{Tr}(x)}{p} \right\rangle \right) \chi(x),  $$ 
$$ |\mathfrak{C}^+_{ns}(p)| = \displaystyle \frac{24}{(p-1)\gcd(12,p+1)}\prod_{j=1}^{\frac{p-3}{2}}\frac{p}{2}B_{2,\omega^{(2p+2)j}} =
 $$  
 $$ = \displaystyle\frac{ 576 \left| \det\left[\displaystyle\frac{p}{2}\left(\displaystyle\sum_{l=0}^{p}B_2\left(  \left\langle \frac{\frac{1}{2}\mbox{Tr}(v^{i-j+l\frac{p-1}{2}}) }{p} \right\rangle \right) - \frac{p+1}{6} \right) \right]_{1\le i,j \le \frac{p-1}{2}} \right|}{(p-1)^2 p (p+1)\gcd(12,p+1)}. $$
\\
\noindent
In the following table we show the factorization of the orders of cuspidal divisor class groups  $ \mathfrak{C}^+_{ns}(p)  $ for some primes $ p \le 101 $:  

\begin{tab}\label{tab} 
\noindent 
\\
\begin{tabular}{rc}
\toprule
$ p $ &  $ |\mathfrak{C}^+_{ns}(p)|  $ \\
\midrule
5 & $ 1 $ \\ 7 & $ 1 $ \\ 11 & $ 11 $ \\ 
13 &  $ 7 \cdot 13^2 $ \\
17 & $ 2^4 \cdot 3 \cdot 17^3 $ \\
19 & $ 3 \cdot 19^3 \cdot 487 $ \\
23 & $ 23^4 \cdot 37181 $ \\
29 & $  2^6 \cdot 5 \cdot 7^2 \cdot 29^6 \cdot 43^2 $\\
31 & $ 2^2 \cdot 5 \cdot 7 \cdot 11 \cdot 31^6 \cdot 2302381 $ \\
37 & $ 3^4 \cdot 7^2 \cdot 19^3 \cdot 37^8 \cdot 577^2 $ \\
41 & $ 2^6 \cdot 5^2 \cdot 7 \cdot 31^4 \cdot 41^9 \cdot 431^2 $ \\
43 & $ 2^2 \cdot 19 \cdot 29 \cdot 43^9 \cdot 463 \cdot 1051 \cdot 416532733 $ \\ 
53 & $ 3^2 \cdot 13^2 \cdot 53^{12} \cdot 96331^2 \cdot 379549^2 $ \\
59 & $ 59^{14} \cdot 9988553613691393812358794271  $ \\
67 & $  67^{16} \cdot 193 \cdot 661^2 \cdot 2861 \cdot 8009 \cdot 11287 \cdot 9383200455691459 $ \\
71 & $ 31 \cdot 71^{16} \cdot 113 \cdot 211 \cdot 281 \cdot 701^2 \cdot 12713 \cdot 13070849919225655729061  $ \\
73 & $ 2^2 \cdot 3^4 \cdot 11^2 \cdot 37 \cdot 73^{17} \cdot 79^2 \cdot 241^2 \cdot 3341773^2 \cdot 11596933^2  $ \\
83 & $ 83^{19} \cdot 17210653 \cdot 151251379  \cdot 18934761332741 \cdot 48833370476331324749419  $ \\
89 & $ 2^2 \cdot 3 \cdot 5 \cdot 11^2 \cdot 13^2 \cdot 89^{21} \cdot 4027^2 \cdot 262504573^2 \cdot 15354699728897^2  $\\
101 & $ 5^4 \cdot 17 \cdot 101^{24} \cdot 52951^2 \cdot 54371^2 \cdot 58884077243434864347851^2 $
\end{tabular}
\end{tab}

\noindent \\
We recall the following result of \cite{BB}:

\begin{thm}\label{Hurw} 
 The genera of $ X^+_{ns}(p) $ are:
 $$  g(X^+_{ns}(p)) = \displaystyle\frac{1}{24}\left( p^2 - 10 p + 23 + 6\left(\frac{-1}{p}\right) + 4\left(\frac{-3}{p}\right) \right). $$
\end{thm} 
\begin{proof}
It is a consequence of Hurwitz's formula \cite[Proposition 1.40]{Shimura:af}:

$$ g(\Gamma) = 1 + \frac{d}{12} - \frac{e_2}{4} -  \frac{e_3}{3} - \frac{e_\infty}{2}.  $$
\noindent 
In this case:  
$ d:= [SL_2(\mathbb{Z}):\Gamma_{ns}^+(p)] = \frac{p(p-1)}{2} $, 
$ e_\infty = \frac{p-1}{2} $ is the number of cusps (see Proposition \ref{numerocuspidi}), $ e_2 $ and $ e_3 $ denote the number of elliptic points of period 2 and 3. We have (cfr. \cite[Proposition 12]{Rebolledo}):  
$$ e_2 = \frac{p+1}{2} - \left(\frac{-1}{p}\right)   \mbox{ and } e_3 =  \frac{1}{2} - \frac{1}{2}\left(\frac{-3}{p}\right). $$
\end{proof} 
 By Theorem \ref{Hurw} we have $ g(X^+_{ns}(5))=g(X^+_{ns}(7))=0 $ so it will not be surprising to find out that $ \mathfrak{C}^+_{ns}(5)$ and $ \mathfrak{C}^+_{ns}(7) $ are trivial. 
 
 For $ 11 \le p \le 31 $ we provide further corroborative evidence of Table \ref{tab}. From \cite[p. 195]{SerreMordel} we have:
 
\begin{thm}\label{x1}
   The modular curve $ X^+_{ns}(p) $ associated to the subgroup $ C_{ns}^+(p) $ is a projective non-singular modular curve which can be defined over $ \mathbb{Q} $. The cusps are defined over $ \mathbb{Q}(\cos(\frac{2\pi}{p})) $, the maximal real subfield of the $ p $-th cyclotomic field.
\end{thm} 
   
From \cite{Chen} we have the following result:
 
\begin{thm}\label{x2}
The jacobian of $ X^+_{ns}(p) $ is isogenous to the new part of the Jacobian $ J_0^+(p^2) $ of $ X_0^+(p^2) $.
\end{thm} 

From \cite[Chapter 12]{Ribet} we have this interesting corollary of the Eichler-Shimura relation \cite[pag. 354]{Diamond:mf}:

\begin{thm}\label{x3} Let $q$ be a prime that does not divide $ N $ and let $f(x) $ the characteristic polynomial of the Hecke operator $ T_q $ acting on $ S^{new}_2(\Gamma^+_0(N)) $. Then:
$$ |{J_0^+}^{new}(N)(\mathbb{F}_q)| = f(q+1).$$
\end{thm}  
\noindent 
Choose a prime $ q \equiv \pm 1 \mbox{ mod }p$ that does not divide $  |\mathfrak{C}^+_{ns}(p)| $.  From the previous theorems, the cuspidal divisor class group $ \mathfrak{C}^+_{ns}(p) $ injects into $ J^+_{ns}(p)(\mathbb{F}_q) $. So we expect that $ |\mathfrak{C}^+_{ns}(p)|  $ divides $ |{J_0^+}^{new}(p^2)(\mathbb{F}_q)| = f_{q,p^2}(q+1)$ where $ f_{q,p^2} $ is the characteristic polynomial of the Hecke operator $T_q $ acting on $ S_2^{new}(\Gamma^+_0(p^2))$. 
From the modular form database of W.Stein we have:\\ 

\noindent $ |{J_0^+}^{new}(11^2)(\mathbb{F}_{23})| = f_{23,121}(24)= 3 \cdot 11 ,$\\
$ |{J_0^+}^{new}(11^2)(\mathbb{F}_{43})| = f_{43,121}(44)= 2^2 \cdot 11 ,$\\
$ |{J_0^+}^{new}(11^2)(\mathbb{F}_{67})| = f_{67,121}(68)= 5 \cdot 11 ,$\\
$ |{J_0^+}^{new}(11^2)(\mathbb{F}_{89})| = f_{89,121}(90)= 3^2 \cdot 11 ,$\\
$ |{J_0^+}^{new}(11^2)(\mathbb{F}_{109})| = f_{109,121}(110)= 2 \cdot 5 \cdot 11 ,$\\
$ |{J_0^+}^{new}(11^2)(\mathbb{F}_{131})| = f_{131,121}(132)= 2^2 \cdot 3 \cdot 11 ,$\\
$ |{J_0^+}^{new}(11^2)(\mathbb{F}_{197})| = f_{197,121}(198)= 2 \cdot 3^2 \cdot 11 ,$\\
$ |{J_0^+}^{new}(11^2)(\mathbb{F}_{199})| = f_{199,121}(200)= 2^2 \cdot 5 \cdot 11 ,$\\
$ |{J_0^+}^{new}(11^2)(\mathbb{F}_{241})| = f_{241,121}(242)= 2 \cdot 11^2 ,$\\
$ |{J_0^+}^{new}(11^2)(\mathbb{F}_{263})| = f_{263,121}(264)= 2^3 \cdot 3 \cdot 11 ,$\\
$ |{J_0^+}^{new}(11^2)(\mathbb{F}_{307})| = f_{307,121}(308)= 2^2 \cdot 7 \cdot 11 ,$\\
$ |{J_0^+}^{new}(11^2)(\mathbb{F}_{331})| = f_{331,121}(332)= 3^3 \cdot 11 ,$\\
$ |{J_0^+}^{new}(11^2)(\mathbb{F}_{353})| = f_{353,121}(354)= 3 \cdot 11^2 ,$\\
$ |{J_0^+}^{new}(11^2)(\mathbb{F}_{373})| = f_{373,121}(374)= 2 \cdot 11 \cdot 17 ,$\\
$ |{J_0^+}^{new}(11^2)(\mathbb{F}_{397})| = f_{397,121}(398)= 2^2 \cdot 3^2 \cdot 11 ,$\\
$ |{J_0^+}^{new}(11^2)(\mathbb{F}_{419})| = f_{419,121}(420)= 2^2 \cdot 3^2 \cdot 11 ,$\\
$ |{J_0^+}^{new}(11^2)(\mathbb{F}_{439})| = f_{439,121}(440)= 2^3 \cdot 5 \cdot 11 ,$\\
$ |{J_0^+}^{new}(11^2)(\mathbb{F}_{461})| = f_{461,121}(462)= 2 \cdot 3 \cdot 7 \cdot 11 ,$\\
$ |{J_0^+}^{new}(11^2)(\mathbb{F}_{463})| = f_{463,121}(464)= 3^2 \cdot 5 \cdot 11 ,$\\

\noindent $ |{J_0^+}^{new}(13^2)(\mathbb{F}_{53})| = f_{53,169}(54)= 7 \cdot 13 ^2 \cdot 127 ,$\\
$ |{J_0^+}^{new}(13^2)(\mathbb{F}_{79})| = f_{79,169}(80)= 7 \cdot 13 ^2 \cdot 449 ,$\\
$ |{J_0^+}^{new}(13^2)(\mathbb{F}_{103})| = f_{103,169}(104)= 7 \cdot 13 ^2 \cdot 967 ,$\\
$ |{J_0^+}^{new}(13^2)(\mathbb{F}_{131})| = f_{131,169}(132)= 7 \cdot 13 ^5, $\\
$ |{J_0^+}^{new}(13^2)(\mathbb{F}_{157})| = f_{157,169}(158)= 7^2 \cdot 13 ^2 \cdot 503,$\\
$ |{J_0^+}^{new}(13^2)(\mathbb{F}_{181})| = f_{181,169}(182)= 7 \cdot 13 ^2 \cdot 4327, $\\
$ |{J_0^+}^{new}(13^2)(\mathbb{F}_{233})| = f_{233,169}(234)= 7 \cdot 13 ^2 \cdot 11731, $\\
$ |{J_0^+}^{new}(13^2)(\mathbb{F}_{311})| = f_{311,169}(312)= 7 \cdot 13 ^2 \cdot 26249, $\\
$ |{J_0^+}^{new}(13^2)(\mathbb{F}_{313})| = f_{313,169}(314)= 7 \cdot 13 ^2 \cdot 29443, $\\
$ |{J_0^+}^{new}(13^2)(\mathbb{F}_{337})| = f_{337,169}(338)= 7 \cdot 13 ^2 \cdot 35449, $\\
$ |{J_0^+}^{new}(13^2)(\mathbb{F}_{389})| = f_{389,169}(390)= 2^3 \cdot 7 \cdot 13 ^2 \cdot 71 \cdot 83, $\\
$ |{J_0^+}^{new}(13^2)(\mathbb{F}_{443})| = f_{443,169}(444)= 2^3 \cdot 7 \cdot 13 ^3 \cdot 643, $\\
$ |{J_0^+}^{new}(13^2)(\mathbb{F}_{467})| = f_{467,169}(468)=  7 \cdot 13 ^2 \cdot 93199, $\\

\noindent $ |{J_0^+}^{new}(17^2)(\mathbb{F}_{67})| = f_{67,289}(68)= 2^8 \cdot 3 \cdot 17^5 \cdot 71  ,$\\
$ |{J_0^+}^{new}(17^2)(\mathbb{F}_{101})| = f_{101,289}(102)= 2^4 \cdot 3^2 \cdot 7 \cdot 17^3 \cdot 19 \cdot 79 \cdot 181 ,$\\
$ |{J_0^+}^{new}(17^2)(\mathbb{F}_{103})| = f_{103,289}(104)= 2^7 \cdot 3^4 \cdot 17^4 \cdot 1601  ,$\\
$ |{J_0^+}^{new}(17^2)(\mathbb{F}_{137})| = f_{137,289}(138)= 2^6 \cdot 3^8 \cdot 17^4 \cdot 181  ,$\\
$ |{J_0^+}^{new}(17^2)(\mathbb{F}_{239})| = f_{239,289}(240)= 2^8 \cdot 3^2 \cdot 17^3 \cdot 373 \cdot 48871  ,$\\
$ |{J_0^+}^{new}(17^2)(\mathbb{F}_{271})| = f_{271,289}(272)= 2^5 \cdot 3^9 \cdot 5^3 \cdot 17^4 \cdot 53  ,$\\
$ |{J_0^+}^{new}(17^2)(\mathbb{F}_{307})| = f_{307,289}(308)= 2^6 \cdot 3^5 \cdot 5 \cdot 17^3 \cdot 23 \cdot 71 \cdot 1423  ,$\\
$ |{J_0^+}^{new}(17^2)(\mathbb{F}_{373})| = f_{373,289}(374)= 2^4 \cdot 3^4 \cdot 17^3 \cdot 23 \cdot 73 \cdot 101 \cdot 2789  ,$\\
$ |{J_0^+}^{new}(17^2)(\mathbb{F}_{409})| = f_{409,289}(410)= 2^7 \cdot 3^5 \cdot 17^3 \cdot 23 \cdot 53 \cdot 71 \cdot 359  ,$\\
$ |{J_0^+}^{new}(17^2)(\mathbb{F}_{443})| = f_{443,289}(444)= 2^5 \cdot 3^2 \cdot 13 \cdot 17^4 \cdot 19 \cdot 79 \cdot 15263  ,$\\

\noindent $ |{J_0^+}^{new}(19^2)(\mathbb{F}_{37})| = f_{37,361}(38)= 2 \cdot 3 \cdot 19^3 \cdot 37 \cdot 487 \cdot 5441 ,$\\
$ |{J_0^+}^{new}(19^2)(\mathbb{F}_{113})| = f_{113,361}(114)= 2^5 \cdot 3^7 \cdot 19^7 \cdot 487 ,$\\
$ |{J_0^+}^{new}(19^2)(\mathbb{F}_{151})| = f_{151,361}(152)= 2^3 \cdot 3^3 \cdot 17 \cdot 19^4 \cdot 487 \cdot 1459141 ,$\\
$ |{J_0^+}^{new}(19^2)(\mathbb{F}_{191})| = f_{191,361}(192)= 3^2 \cdot 11^5 \cdot 19^6 \cdot 73 \cdot 487 ,$\\
$ |{J_0^+}^{new}(19^2)(\mathbb{F}_{227})| = f_{227,361}(228)= 2^2 \cdot 3^4 \cdot 19^3 \cdot 487 \cdot 971 \cdot 7323581,$\\
$ |{J_0^+}^{new}(19^2)(\mathbb{F}_{229})| = f_{229,361}(230)= 3 \cdot 11 \cdot 17 \cdot 19^3 \cdot 467 \cdot 487 \cdot 2819^2 ,$\\
$ |{J_0^+}^{new}(19^2)(\mathbb{F}_{379})| = f_{379,361}(380)= 2^6 \cdot 3 \cdot 5^2 \cdot 19^3 \cdot 179 \cdot 487 \cdot 4019 \cdot 33247 ,$\\
$ |{J_0^+}^{new}(19^2)(\mathbb{F}_{419})| = f_{419,361}(420)= 2^6 \cdot 3^2 \cdot 5^3 \cdot 19^3 \cdot 487 \cdot 509^2 \cdot 16487 ,$\\
$ |{J_0^+}^{new}(19^2)(\mathbb{F}_{457})| = f_{457,361}(458)= 2^4 \cdot 3 \cdot 5^4 \cdot 19^3 \cdot 487 \cdot 521^2 \cdot 65629 ,$\\

\noindent $ |{J_0^+}^{new}(23^2)(\mathbb{F}_{47})| = f_{47,529}(48)= 2^3 \cdot 3^3 \cdot 7^4 \cdot 11 \cdot 13 \cdot 23^4 \cdot 8117 \cdot 37181 ,$\\
$ |{J_0^+}^{new}(23^2)(\mathbb{F}_{137})| = f_{137,529}(138)= 2^4 \cdot 3^6 \cdot 23^8 \cdot 2399 \cdot 37181 \cdot 75553 ,$\\
$ |{J_0^+}^{new}(23^2)(\mathbb{F}_{139})| = f_{139,529}(140)= 2^4 \cdot 3^8 \cdot 23^9 \cdot 107^2 \cdot 109 \cdot 37181 ,$\\
$ |{J_0^+}^{new}(23^2)(\mathbb{F}_{229})| = f_{229,529}(230)= 2^6 \cdot 11 \cdot 23^6 \cdot 43 \cdot 67 \cdot 37181 \cdot 325729 \cdot 1296721 ,$\\
$ |{J_0^+}^{new}(23^2)(\mathbb{F}_{277})| = f_{277,529}(278)= 2^8 \cdot 3^{10} \cdot 23^7 \cdot 113^2 \cdot 331 \cdot 7193 \cdot 37181 ,$\\
$ |{J_0^+}^{new}(23^2)(\mathbb{F}_{367})| = f_{367,529}(368)= 2^4 \cdot 23^5 \cdot 67^2 \cdot 193 \cdot 1847 \cdot 37181 \cdot 44617  \cdot    8643209 ,$\\
$ |{J_0^+}^{new}(23^2)(\mathbb{F}_{461})| = f_{461,529}(462)=  3^6 \cdot 7^4 \cdot  23^7 \cdot  43^2 \cdot  67 \cdot 199 \cdot 2857^2 \cdot 37181 ,$\\

\noindent $ |{J_0^+}^{new}(29^2)(\mathbb{F}_{59})| = f_{59,841}(60)= 2^8 \cdot 3^2 \cdot 5 \cdot 7^2 \cdot 11^2 \cdot 17 \cdot 23^2 \cdot 29^6 \cdot 43^2 \cdot 569 \cdot 967^2 \cdot 2999 \cdot 11695231 ,$\\
$ |{J_0^+}^{new}(29^2)(\mathbb{F}_{173})| = f_{173,841}(174)= 2^{10} \cdot 3^2 \cdot 5^2 \cdot 7^2 \cdot 29^6 \cdot 31 \cdot 41^2 \cdot 43^2 \cdot 89 \cdot 419^2 \cdot 719 \cdot 1061 \cdot 36571 \cdot 1269691 \cdot 1909421 ,$\\
$ |{J_0^+}^{new}(29^2)(\mathbb{F}_{233})| = f_{233,841}(234)= 2^{10} \cdot 3^2 \cdot 5 \cdot 7^2 \cdot 29^6 \cdot 43^2 \cdot 167^2 \cdot 211^2 \cdot  421 \cdot 1049 \cdot 3989 \cdot 317321 \cdot 422079165281099 ,$\\
$ |{J_0^+}^{new}(29^2)(\mathbb{F}_{347})| = f_{347,841}(348)= 2^8 \cdot 3^{12}  \cdot 5^3 \cdot 7^2 \cdot 11 \cdot 23^2 \cdot 29^6 \cdot 31 \cdot 43^2 \cdot 71 \cdot 127^2 \cdot 967^2 \cdot 9601 \cdot 783719 \cdot 7292986801 ,$\\
$ |{J_0^+}^{new}(29^2)(\mathbb{F}_{349})| = f_{349,841}(350)= 2^8 \cdot 5^9 \cdot 7^2 \cdot 13^2 \cdot 19 \cdot 23 \cdot 29^7 \cdot 43^2 \cdot 83^2 \cdot 103 \cdot 211 \cdot 3786151 \cdot 92610181 \cdot 3477902249 ,$\\
 $ |{J_0^+}^{new}(29^2)(\mathbb{F}_{463})| = f_{463,841}(464)= 2^{13} \cdot 5^7 \cdot 7^7 \cdot 29^6 \cdot 43^3 \cdot 59 \cdot 97^3 \cdot 461^3 \cdot  1459 \cdot 23656223369 \cdot 230667656992649 ,$\\

\noindent $ |{J_0^+}^{new}(31^2)(\mathbb{F}_{61})| = f_{61,961}(62) = 2^{10} \cdot 5 \cdot 7 \cdot 11 \cdot 31^{7} \cdot 137 \cdot 179 \cdot 1249 \cdot 10369 \cdot 26699 \cdot 38177 \cdot 2302381 \cdot 24080801 ,$\\
$ |{J_0^+}^{new}(31^2)(\mathbb{F}_{311})| = f_{311,961}(312)= 2^8 \cdot 3^2 \cdot 5 \cdot 7^2 \cdot 11 \cdot 31^7 \cdot 409 \cdot 3793^2 \cdot 51551^2 \cdot 162691 \cdot 2302381 \cdot 22340831^2 \cdot 24037019 ,$\\
$ |{J_0^+}^{new}(31^2)(\mathbb{F}_{373})| = f_{373,961}(374)= 
2^4 \cdot 5 \cdot 7^2 \cdot 11^2 \cdot 13^2 \cdot 31^6 \cdot 251 \cdot 449 \cdot 2302381 \cdot 366424077359 \cdot 13600706515978033^2 ,$\\
$ |{J_0^+}^{new}(31^2)(\mathbb{F}_{433})| = f_{433,961}(434)= 2^6 \cdot 3^6 \cdot 5 \cdot 7 \cdot 11 \cdot 17 \cdot 31^{11}  \cdot 89 \cdot 97 \cdot 191 \cdot 401 \cdot 1153 \cdot 54331 \cdot 126961 \cdot 2302381 \cdot 12958271 \cdot 53053053405791.$\\

\noindent For $ 11 \le p \le 23 $ we have:
$$ \mathop
\textbf{\mbox{ gcd }}_{ \begin{scriptsize} \begin{array}{c} q < 500 \mbox{ prime},   \\  q \equiv \pm 1 \mbox{ mod }p \end{array} \end{scriptsize} } |{J_0^+}^{new}(p^2)(\mathbb{F}_{q})| = |\mathfrak{C}^+_{ns}(p)|.  $$
For $ p=29 $ and $ p=31 $ we have: 
$$ \mathop\textbf{\mbox{ gcd }}_{ \begin{scriptsize} \begin{array}{c} q < 500 \mbox{ prime},   \\  q \equiv \pm 1 \mbox{ mod }p \end{array} \end{scriptsize} } |{J_0^+}^{new}(p^2)(\mathbb{F}_{q})| = 4|\mathfrak{C}^+_{ns}(p)|.$$
We can improve the result by using the isogeny (cfr.\cite[Paragraph 6.6]{Diamond:mf}):
$$ {J_0^+}^{new}(p^2) \longrightarrow \mathop{\bigoplus_{f}} A'_{p,f}  $$ 
where the sum is taken over the equivalence classes of newforms $ f\in  S_2(\Gamma^+_0(p^2)) $. Two forms $ f $ and $ g $ are declared equivalent if $ g=f^{\sigma} $ for some automorphism $ \sigma:  \mathbb{C} \longrightarrow \mathbb{C}$. Denote with $ \mathbb{K}_f $ the number field of $ f $. We have:
$$ \mathop
\textbf{\mbox{ gcd }}_{ \begin{scriptsize} \begin{array}{c} q < 500 \mbox{ prime},   \\  q \equiv \pm 1 \mbox{ mod }29 \end{array} \end{scriptsize} } |A'_{29,f_1}(\mathbb{F}_{q})| = 7^2 \mbox{ where }\mathbb{K}_{f_1}= \mathbb{Q}(\sqrt{2}), $$
$$ \mathop
\textbf{\mbox{ gcd }}_{ \begin{scriptsize} \begin{array}{c} q < 500 \mbox{ prime},   \\  q \equiv \pm 1 \mbox{ mod }29 \end{array} \end{scriptsize} } |A'_{29,f_2}(\mathbb{F}_{q})| = 29 \mbox{ where }\mathbb{K}_{f_2}= \mathbb{Q}(\sqrt{5}), $$
$$ \mathop
\textbf{\mbox{ gcd }}_{ \begin{scriptsize} \begin{array}{c} q < 500 \mbox{ prime},   \\  q \equiv \pm 1 \mbox{ mod }29 \end{array} \end{scriptsize} } |A'_{29,f_3}(\mathbb{F}_{q})|= \mathop
\textbf{\mbox{ gcd }}_{ \begin{scriptsize} \begin{array}{c} q < 500 \mbox{ prime},   \\  q \equiv \pm 1 \mbox{ mod }29 \end{array} \end{scriptsize} } |A'_{29,f_4}(\mathbb{F}_{q})|=   2^3 \cdot 43 $$   where  $\mathbb{K}_{f_3}= \mathbb{K}_{f_4} $ and $ [\mathbb{K}_{f_3}:\mathbb{Q}]=3$, 
  $$ \mathop
\textbf{\mbox{ gcd }}_{ \begin{scriptsize} \begin{array}{c} q < 500 \mbox{ prime},   \\  q \equiv \pm 1 \mbox{ mod }29 \end{array} \end{scriptsize} } |A'_{29,f_5}(\mathbb{F}_{q})|= 5 \cdot 29^2 \mbox { where } [\mathbb{K}_{f_5}:\mathbb{Q}]=6, $$
$$ \mathop
\textbf{\mbox{ gcd }}_{ \begin{scriptsize} \begin{array}{c} q < 500 \mbox{ prime},   \\  q \equiv \pm 1 \mbox{ mod }29 \end{array} \end{scriptsize} } |A'_{29,f_6}(\mathbb{F}_{q})|= 29^3 \mbox { where } [\mathbb{K}_{f_6}:\mathbb{Q}]=8, $$
$$ \mathop
\textbf{\mbox{ gcd }}_{ \begin{scriptsize} \begin{array}{c} q < 500 \mbox{ prime},   \\  q \equiv \pm 1 \mbox{ mod }31 \end{array} \end{scriptsize} } |A'_{31,g_1}(\mathbb{F}_{q})| = 2^2 \cdot 7 \mbox{ where }\mathbb{K}_{g_1}= \mathbb{Q}(\sqrt{2}), $$
$$ \mathop
\textbf{\mbox{ gcd }}_{ \begin{scriptsize} \begin{array}{c} q < 500 \mbox{ prime},   \\  q \equiv \pm 1 \mbox{ mod }31 \end{array} \end{scriptsize} } |A'_{31,g_2}(\mathbb{F}_{q})| = 5 \cdot 11 \mbox{ where }\mathbb{K}_{g_2}= \mathbb{Q}(\sqrt{5}), $$
$$ \mathop
\textbf{\mbox{ gcd }}_{ \begin{scriptsize} \begin{array}{c} q < 500 \mbox{ prime},   \\  q \equiv \pm 1 \mbox{ mod }31 \end{array} \end{scriptsize} } |A'_{31,g_3}(\mathbb{F}_{q})|= 2302381 \mbox { where } [\mathbb{K}_{g_3}:\mathbb{Q}]=8, $$
$$ \mathop  
\textbf{\mbox{ gcd }}_{ \begin{scriptsize} \begin{array}{c} q < 500 \mbox{ prime},   \\  q \equiv \pm 1 \mbox{ mod }31 \end{array} \end{scriptsize} } |A'_{31,g_4}(\mathbb{F}_{q})|= 31^6 \mbox { where } [\mathbb{K}_{g_4}:\mathbb{Q}]=16. $$
So for $ p=29 $ and $ p=31 $ we have:
$$ \prod_f \mathop
\textbf{\mbox{ gcd }}_{ \begin{scriptsize} \begin{array}{c} q < 500 \mbox{ prime},   \\  q \equiv \pm 1 \mbox{ mod }p \end{array} \end{scriptsize} } |A'_{p,f}(\mathbb{F}_{q})|= |\mathfrak{C}^+_{ns}(p)|  $$
where the product runs over all equivalence classes of newforms.

\subsection*{Acknowledgements} I would like to express my gratitude to my advisor  Prof. René Schoof for his valuable remarks during the development of this research work, especially for the last section.

\end{document}